\documentclass[12pt]{article}
\usepackage{amsmath,amssymb,amsthm,amsfonts}

\usepackage{amsmath,amssymb,amsfonts}
\usepackage{graphicx}

\def\RR{{{\mathbb R}}}
\def\ZZ{\mathbb Z} 
\def\CC{\mathbb C} 
 
\def\phi{\varphi}
\def\ba{{\mathbf a}}
\def\bc{{\mathbf c}}

\def\blambda{{\mathbf \Lambda}}

\def\bPhi{{\mathbf \Phi}}
\def\bPsi{{\mathbf \Psi}}
\def\dprod#1#2{{\displaystyle{\prod_{#1}^{#2}}}}
\def\supp{{\rm supp}\ }

\newtheorem{theorem}{Theorem}[section]

\newtheorem{corollary}[theorem]{Corollary}

\numberwithin{equation}{section}

\title{Bell-shaped Nonstationary Refinable Ripplets}
\author{Francesca Pitolli 
\thanks{{\it Dept. SBAI, University of Roma ''La Sapienza''}, 
Via A. Scarpa 16, 00161 Roma, Italy  
e-mail: \tt{francesca.pitolli@sbai.uniroma1.it}} 
\\ \\ {\it Dedicated to Laura Gori on the occasion of her 80th birthday}
}

\date{}

\begin{document}

\maketitle

\begin{abstract}
We study the approximation properties of the class of nonstationary refinable ripplets introduced in \cite{GP08}.
These functions are solution of an infinite set of nonstationary refinable equations and are 
defined through sequences of scaling masks that have an explicit expression.
Moreover, they are variation-diminishing and highly localized in the scale-time plane, properties that make them particularly 
attractive in applications. Here, we prove that they enjoy Strang-Fix conditions and 
convolution and differentiation rules and that they are bell-shaped.
Then, we construct the corresponding minimally supported nonstationary prewavelets and 
give an iterative algorithm to evaluate the prewavelet masks. Finally, we give a procedure to construct the associated nonstationary biorthogonal bases and filters to be used in efficient decomposition and reconstruction algorithms. 

As an example, we calculate the prewavelet masks and the nonstationary biorthogonal filter pairs corresponding to the $C^2$ nonstationary scaling functions in the class and construct the corresponding prewavelets and biorthogonal bases. 
A simple test showing their good performances in the analysis of a spike-like signal is also presented. 
\end{abstract}

\maketitle

\section{Introduction}
A {\em ripplet} is a function $f$ whose integer translate are {\em totally positive} \cite{Ka}, 
i.e.  for any ordered real numbers $x_1<\cdots < x_r$, and any 
ordered integers $\alpha_1 < \cdots < \alpha_r$, $r \ge 1$, it holds 
\begin{equation}
{\rm det} \,\bigl ( f(x_i-\alpha_\ell) \bigr )_{1\le i,\ell \le r}\ge 0.
\end{equation} 
Total positivity implies that the integer translates of $f$ are {\em variation diminishing}, i.e.
for any finite sequence ${\mathbf c}={\{c_\alpha\}}$
\begin {equation} 
S^-\bigl(\sum_{\alpha }\, c_\alpha \,f(\cdot-\alpha) \bigr)\le S^-(\mathbf c)\,,
\label {var.dim.1}
\end {equation}
where $S^-$ denotes the strict sign changes of its argument.
The disequality (\ref{var.dim.1}) in turn implies that the system $\{f(\cdot-\alpha)\}$ has shape-preserving properties, which are known to play a crucial role in several applications, from approximation of data to CAGD \cite{GM96,Mi95}.

The concept of a ripplet was  first introduced by  Goodman and  Micchelli in \\ \cite{GM92}, where
the authors focused their interest on {\em two-scale refinable ripplets}, i.e. ripplets that are solution of a 
{\em two-scale refinable equation}
\begin{equation}  \label{stazref}
\varphi = \sum_{\alpha } \, a_\alpha \, \varphi(2\cdot-\alpha)\,,
\end{equation}
where the {\em scaling mask} $\ba=\{a_\alpha\}$ is a suitable real sequence.
Well known examples of refinable ripplets are the cardinal B-splines, i.e. the polynomial B-splines on integer nodes.
Starting from the seminal paper of Goodman and Micchelli, many families of two-scale refinable ripplets were  
constructed (see, for instance, \cite{CW97}, \cite{GP00}). More recently,  
$M$-scale refinable ripplets, with dilation $M$ greater than 2, 
were addressed in \cite{GS04} and refinable ripplets with dilation 3 
were constructed in \cite{GPS11} (see also \cite{GPS13}).

The interest in refinable ripplets lies in the fact that they give rise to 
convergent subdivision algorithms for the reconstruction of curves and the limit curves they generate 
preserve the shape of the initial data \cite{Go96,Mi95}. Refinable ripplets can also be proved to solve 
the cardinal interpolation problem \cite{Mi95}. For instance, the construction of cardinal interpolants 
by by means of the refinable ripplets in \cite{GP00} was addressed in \cite{Pi98}.
 
Refinable ripplets have good properties not only in the context of geometric modeling and function approximation  
but they also enjoy some optimality properties useful in signal processing. 
In fact, refinable ripplets induces a multiresolution analysis in $L_2(\RR)$ that allows us to generate
a nested sequence of wavelet spaces. Actually, it is always possible to construct compactly supported semiorthogonal wavelet bases 
starting from a refinable ripplet \cite{Mi91}.
Moreover, refinable ripplets have asymptotically, i.e. when their smoothness tends to infinity,
the same optimal time-frequency window achieved by the Gaussian function \cite{CW97}.
Since the rate of convergence  can be proved to be very fast for a large class of ripplets,
including, for instance, the refinable ripplets introduced in \cite{GP00},
refinable ripplets can approximate
the Gaussian with high accuracy giving rise to efficient algorithms for signal analysis \cite{CGL04}.
Finally, a ripplet can be seen as a discrete kernel
satisfying a causality property so  making the refinable ripplets particularly attractive 
in the  scale-time analysis of signals \cite{GGL07}.

All the refinable ripplets quoted above are {\em stationary} in the sense that
they satisfy the functional equation (\ref{stazref}) with the same mask sequence at each dyadic scale. 
For this reason usually (\ref{stazref}) is referred to as a {\em stationary two-scale equation}.
From the point of view of signal processing this means that the same 
analysis and synthesis filters are used at all dyadic scales \cite{SN96}. 
Nevertheless, the use of the same set of filters at each scale does not give great flexibility in applications,
especially when some preprocessing steps with different filters are required.
From the functional point of view, the use of different filters at different scales gives rise to a {\em nonstationary multiresolution analysis}
(see, for instance, \cite{BDR93}, \cite{GL99}, \cite{Ha10}, \cite{NPS11}).
A nonstationary multiresolution analysis can be generated by a set of {\em nonstationary refinable functions}, i.e. an infinite set of functions 
$\{\phi^m: m \in \ZZ_+\}$ satisfying an infinite set of {\em nonstationary two-scale equations}
\begin{equation} 
\phi^m=\sum_{\alpha\in \ZZ}\ a_{\alpha}^{m}\ \phi^{m+1}(\cdot-2^{-(m+1)}\alpha),   \quad m\in \ZZ_+\,,  
\label{nseq_int}
\end{equation}
for some infinite sequence of masks $\{\ba^m: m \in \ZZ_+\}$, 
each mask $\ba^m=\{a_\alpha^m\}$ being different at each dyadic scale.

The use of different scaling masks at different scales allows us to construct refinable functions endowed with properties that cannot be achieved in the stationary setting. For instance,
the nonstationary process generated by the B-spline masks of increasing support gives rise to a refinable function that is compactly supported while belongs to $C^{\infty}(\RR)$ \cite{DR95,NPS11}. 
Other families of $C^{\infty}(\RR)$ nonstationary refinable functions based on 
pseudo-spline masks were constructed in \cite{HZ08}.
Exponential splines too can be associated to a nonstationary process \cite{BDR93,DL02}.
They reproduce exponential polynomials and the associated wavelet bases can be successfully used in the analysis of signals 
with exponential behavior \cite{VBU07}. 
Families of nonstationary refinable ripplets were introduced in \cite{CGP07} and \cite{GP08}. 
In particular, the latter are highly localized in the scale-time plane, a
property that is crucial in many applications. For this reason, in the present paper we focus our interest
in this family of nonstationary refinable ripplets and prove that 
these ripplets enjoy several properties which are relevant in the context of both geometric modeling and signal processing. 
Then, we construct the associated wavelet bases.
We notice that, since the refinable ripplets we are considering are non orthogonal,  
compactly supported orthogonal wavelets belonging to the space generated by their translates do not exists. 
This motivates us to construct nonstationary {\em semiorthogonal prewavelets}. In fact, giving up the orthogonality
condition we can build wavelet bases with compact support.  
Moreover, we give a procedure to construct the nonstationary compactly supported biorthogonal bases associated with the
nonstationary refinable ripplets we are considering. Since we are interested in implementing efficient nonstationary 
decomposition and reconstruction algorithms for signal processing,  we construct also the corresponding pairs of nonstationary filters 
to be used in the analysis and synthesis of a given discrete signal. 

The paper is organized as follows.
In Section~2 we give some basic definition concerning nonstationary multiresolution analysis and wavelet spaces and 
recall some results about the existence of solutions of the nonstationary refinable 
equations (\ref{nseq_int}). The class of nonstationary refinable functions we are interested in 
is described in Section~3, 
while in Section~4 we proved some approximation properties that were not addressed in \cite{GP08}. 
In Section~5 we construct the nonstationary prewavelet bases associated with the nonstationary refinable ripplets
in the class and give an efficient algorithm to evaluate nonstationary prewavelet masks.  
The construction of compactly supported biorthogonal bases and filters, which give rise to efficient decomposition and reconstruction formulas for discrete signals, is addressed in Section~6.
Finally, in Section~7 some examples of nonstationary masks and refinable bases are given
and the corresponding nonstationary filters are constructed. A simple test on the analysis of a spike-like
signal is also shown.

\section{Nonstationary Multiresolution Analysis and Wavelet Spaces}

Wavelet spaces are usually constructed starting from a multiresolution analysis that is a
sequence $\{V^m\}$ of nested subspaces which are dense in $L_2(\RR)$ and enjoy the separation property.
Thus, the corresponding wavelet space $W^m\subset V^{m+1}$ is defined as the orthogonal complement of 
$V^m$ in $V^{m+1}$  \cite{Dau92}.
In the stationary case all the spaces $\{ V^m \}$ are generated by the dilates and translates of a unique
refinable function $\phi\in V^0$, solution of the stationary two-scale equation (\ref{stazref}), while the spaces $W^m$
are generated by the dilates and translates of a unique {\em wavelet} $\psi\in W^0 \subset V^1$.

In contrast, in the nonstationary setting any space $V^m$ (resp. $W^m$) is generated by the 
$2^{-m}$-shifts of a different refinable function $\phi^m$ (resp. wavelet $\psi^m$), which are not dilates of one another.
As a consequence, the spaces $\{V^m\}$ and $\{W^m\}$ are not scaled versions of the spaces $V^0$ and $W^0$, respectively.

Nonstationary multiresolution analysis are addressed in several papers in the literature
(see, for instance, \cite{BDR93}, \cite{CD96}, \cite{DL02}, \cite{GL99}, \cite{Ha10}, \cite{Ha12}, \cite{NPS11} and references therein). 
Here, following \cite{BDR93} and \cite{NPS11}, we say that a space sequence $\{ V^m: m \in \ZZ_+\}$
forms a nonstationary multiresolution analysis of $L_2(\RR)$ if 

\begin{center}
\medskip
\begin{tabular}{l}
({\it i}) $V^m \subset V^{m+1}$, $m \in \ZZ_+$;  \quad
({\it ii}) $\overline{\cup_{m \in \ZZ_+} V^m} = L_2(\RR)$; \quad
({\it iii}) $\bigcap_{m\in \ZZ_+} V^m = \{0\}$; \\ \\
({\it iv}) for any $m \in \ZZ_+$, there exists a $L_2(\RR)$-stable basis $\varphi^m$ in $V^m$
such that \\
$V^m = \overline{span} \ \left \{ \phi^m (\cdot - 2^{-m}\alpha), \alpha \in \ZZ \right \}\,.
$
\end{tabular}
\medskip
\end{center}

Property ({\it i}) implies that the {\em refinable functions} $\{\phi^m: m\in \ZZ_+\}$ 
satisfy  a set of {\em nonstationary refinable equations}, i.e.
\begin{equation} \label{NSRE}
\phi^m=\sum_{\alpha\in \ZZ}\ a_{\alpha}^{m}\ \phi^{m+1}(\cdot-2^{-(m+1)}\alpha),   \quad m\in \ZZ_+\,,  
\end{equation}
for some sequence of {\em scaling masks} $\{\ba^m: m \in \ZZ_+\}$, where $\ba^m=\{a_\alpha^m: \alpha \in \ZZ\}\in \ell_2(\ZZ)$
and 
\begin{equation}
\sum_{\alpha\in \ZZ} \ a_\alpha^m = 1\,.
\end{equation}
\\
Properties ({\it ii}) and ({\it iii}) are always true if any $\phi^m$ is compactly supported.
As for {\it (iv)}, the function system $\{\phi^m(\cdot-2^{-m}\alpha): \alpha \in \ZZ\}$, $m \in \ZZ_+$,
is $L_2(\RR)$-stable if and only if the Fourier transform of $\phi^m$ has no $2^{m+1} \pi$-periodic real zeros 
(cf. \cite{BDR93}).

The existence of a unique set of functions $\{\phi^m:m\in \ZZ_+\}$ solution to (\ref{NSRE}), as well as their properties, 
are related to the properties of the mask sequence $\{\ba^m:m \in \ZZ_+\}$.
Since we are interested in the the case of compactly supported masks, 
we assume that any mask $\ba^m$ is  compactly supported  with
\begin{equation}
{\rm supp} \, (\ba^m) \subseteq \Omega \subset \ZZ\,, \qquad m \in \ZZ_+\,,
\end{equation}
and there exists a {\em fundamental mask} $\ba=\{a_\alpha: \alpha \in \Omega\}$ 
satisfying the sum rules  
\begin{equation}
\sum_{\alpha\in \ZZ} \ a_{2\alpha+\gamma} =1\,, \qquad \gamma \in \ZZ\,,
\end{equation}
such that
\begin{equation} 
\sum_{m \in \ZZ_+} \, | a^m _\alpha - a_\alpha|<\infty\,, \qquad \alpha \in \Omega\,,
\end{equation}
(cf. \cite{GL99}).
Even if the assumption above excludes the mask sequence associated with the up function and the nonstationary mask sequence
considered in \cite{Ha10}, nevertheless it covers many nonstationary scaling masks, such us those ones associated with the exponential splines \cite{DL02}, \cite{VBU07} and the mask families associated with the ripplets introduced in 
\cite{CGP07}, \cite{GP08}.
For the case of mask sequences with growing support we refer the reader to \cite{CD96}, \cite{NPS11}.

The nonstationary refinable equations (\ref{NSRE}) can be associated 
to a  {\em nonstationary cascade algorithm}  \cite{GL99}, which    
generates at any iteration $k\in \ZZ_+$ the  sequence of functions $\{h^{m}_{k}: m\in \ZZ_+\}$ 
by 
\begin{equation} \label{NSCA}
h^{m}_{k+1}=\sum_{\alpha \in \Omega} \ a_{\alpha}^{m}\ 
h^{m+1}_{k}(\cdot -2^{-(m+1)}\alpha),  \qquad k\in \ZZ_+\,, \quad m\in \ZZ_+\,. 
\end{equation}
Without loss of generality, we assume that the starting function $h^{m}_{0}$ is the same for all $m\in \ZZ_+$, i.e. 
$h^{m}_{0}=h_0$, where $h_0$ is a given $L_2(\RR)$-stable function with ${\widehat h}_{0}(0)=1$,
so that $ {\widehat h}_{k}^m(0)=1$ for any $k,m \in \ZZ_+$.

The convergence of the cascade algorithm is related to the spectral properties of the {\em fundamental transition operator} $T:\ell_0(\ZZ) \to \ell_0(\ZZ)$, defined as
\begin{equation}
(T \ \blambda)_\alpha=  2 \, \sum_{\beta \in \ZZ} \ \check{a}_{2\alpha-\beta} \ \lambda_\beta\,,  \qquad \alpha \in \ZZ,
\end{equation}
where  
\begin{equation}
\check{\ba}=\{\check{a}_\alpha \in \ZZ\}\,, \qquad \check{a}_\alpha= \sum_{\beta \in \ZZ} \ a_\beta \ a_{\beta-\alpha}\,, \qquad 
 \alpha \in \ZZ\,, 
\end{equation}
is the autocorrelation of the fundamental mask $\ba$.
The cascade sequence $\{h^{m}_{k}: k\in \ZZ_+\}$ converges strongly to $\phi^m$ in $L_2(\RR)$ as $k \to \infty$,
uniformly in $m$, if and only if the fundamental transition operator has unit spectral radius, 1 is the unique eigenvalue on the unit circle and is simple \cite[Th. 1.3]{GL99}.
Under these hypotheses the sequence $\{\phi^m: m\in \ZZ_+\}$ converges strongly to the solution of the stationary refinable equation
\begin{equation}
\phi  = 2 \,\sum_{\alpha \in \Omega} \ a_\alpha \ \phi(2\cdot-\alpha)\,.
\end{equation}
We notice that any nonstationary mask sequence $\{\ba^m : m \in \ZZ_+\}$ having a B-spline mask as fundamental mask, 
gives rise to a convergent cascade algorithm \cite{DL02,GL99}.

\medskip

The nonstationary refinable equations  in the Fourier space read
\begin{equation} \label{NSRE_Fou}
\widehat \phi^m (\omega) = A^m\bigl (e^{-i \frac {\omega}{2^{m+1}}} \bigr ) \, \widehat \phi^{m+1} (\omega)\,, \qquad m \in \ZZ_+\,,
\end{equation}
where  the Laurent polynomials
\begin{equation}
A^m(z)=\sum_{\alpha \in \Omega} \, a^m_\alpha \, z^\alpha\,, \qquad m \in \ZZ_+\,, \quad z \in \CC\,,
\end{equation}
are the {\it mask symbols}.
We notice that any refinable function $\phi^m$ is normalized so that $\widehat \phi^m(0)=1$. 
In case the nonstationary cascade algorithm converges, the Fourier transform of $\phi^m$ is given by
\begin{equation} \label{infprod}
\widehat \phi^m (\omega) = \prod_{k=m}^\infty \,  A^k\bigl (e^{-i \frac {\omega}{2^{k+1}}} \bigr ) \,, \qquad \omega \in \RR\,.
\end{equation}

Given a nonstationary multiresolution analysis $\{V^m: m \in \ZZ_+\}$, at any level $m \in \ZZ_+$ 
the {\em wavelet space} 
$W^m$ is defined as the orthogonal complement of $V^m$ in $V^{m+1}$, i.e.
\begin{equation}
W^{m}=V^{m+1}\ominus V^m\,, \qquad m \in \ZZ_+\,.
\end{equation}
Any wavelet space is generated by the $2^{-m}$-shifts of a wavelet function $\psi^m$, i.e.
\begin{equation}
W^m = \overline{span} \, \left\{ \psi^m (\cdot - 2^{-m}\alpha), \alpha \in \ZZ \right \}\,,
\end{equation}
any $\psi^m$ being different at any scale $m$.
The existence of a set of generating wavelets $\{\psi^m, m \in \ZZ_+\}$ is always assured 
when all the refinable functions $\{\phi^m, m \in \ZZ_+\}$ are compactly supported.
Moreover, it is always possible to construct compactly supported
$L_2(\RR)$-stable wave-lets associated with compactly supported $L_2(\RR)$-stable refinable functions 
\cite[\S 4]{BDR93}.

\section{A Class of Bell-shaped Nonstationary Refinable Ripplets}

The class of univariate compactly supported nonstationary masks and refinable functions 
we are interested in was introduced in \cite{GP08}. Let us denoted the masks in the class by  
\begin{equation}
\{\ba^{(n,m)}: m \in \ZZ_+\}\,, \qquad \ba^{(n,m)}=\{a^{(n,m)}_0, \ldots,a^{(n,m)}_{n+1}\}\,, \quad m \in \ZZ_+\,,
\end{equation}
where $n$ is an integer $\ge 2$, related to the support of the mask $\ba^{(n,m)}$. 
\\
The explicit expression of $\ba^{(n,m)}$ is as follows.
For any $n$, the starting mask $\ba^{(n,0)}$ has entries
\begin{equation}  \left \{ \begin{array}{l}
a^{(n,0)}_0=a^{(n,0)}_1=\frac 1 2\,, \\ \\
a^{(n,0)}_\alpha=0\,, \qquad \hbox{otherwise}\,,
\end{array} \right.
\end{equation}
while the entries of the higher level masks $\ba^{(n,m)}$, $m \ge 1$, have expression
\begin{equation}  \left \{ \begin{array}{l}
a^{(n,m)}_\alpha =\displaystyle{\frac{1}{2^{n+1+m^{-\mu}}} \left [{\binom{n+1}{\alpha}}
+ 4(2^{m^{-\mu}}-1) {\binom{n-1}{\alpha-1}} \right ]}, \quad 0 \le\alpha \le n+1\,,\\ \\
a^{(n,m)}_\alpha = 0\,,  \qquad \hbox{otherwise}\,.
\end{array} \right.
\end{equation}
(We assume $\binom{n}{\alpha}=0$ when $\alpha <0$ or $\alpha >n$.)
\\
We notice that  $\mu>1$ is a real parameter that acts as a {\em tension parameter}. 
In fact, the larger $\mu$ the faster the factor $2^{m^{-\mu}}$ in the nonstationary mask coefficients
goes to 1 when $m\to \infty$. This means that the nonstationary process behaves in practice as a stationary one
if we choose large values of $\mu$. Thus, the more interesting cases are obtained for values of $\mu$ that 
are close to 1.
\\
For any $n$ and $\mu$, the mask $\ba^{(n,m)}$ is compactly supported on $[0,n+1]$ and is {\em bell-shaped}, 
i.e. its entries are positive, centrally symmetric 
and strictly increasing on $\bigl[0,\left[\frac{n+1}2\right]\bigr]$.
\\
We notice that for any $n \ge 2$ the 0-level scaling mask $\ba^{(n,0)}$
is the mask of the characteristic function of the interval [0,1], while the
$m$-level scaling masks $\ba^{(n,m)}$, $m >0$, are related to the class of stationary masks introduced in \cite{GP00}.
\\
Any mask sequence $\{\ba^{(n,m)}: m \in \ZZ_+\}$ can be associated with the set of nonstationary refinable equations
\begin{equation} \label{REphinm}
\phi^{(n,m)} = \sum_{\alpha\in \sigma^{(n,m)}_a} \, a^{(n,m)}_\alpha \, \phi^{(n,m+1)} (\cdot-2^{-(m+1)}\alpha)\,, 
\qquad m \in \ZZ_+\,,
\end{equation}
where
\begin{equation}
\sigma^{(n,m)}_a= {\rm supp} \ \bigl ( \ba^{(n,m)} \bigr) = \left \{ \begin{array}{ll}
[0,1], & m=0\,,     \\    
\left[ 0,n+1\right], &  m>0\,.
\end{array} \right.
\end{equation}
From \cite{GP08} it follows that for any $n\ge 2$ and $\mu>1$, $\phi^{(n,m)}$, $m \in \ZZ_+$, is  compactly supported with
\begin{equation} \label{suppphi}
supp \ \phi^{(n,m)}=\bigl [ 0, L_{(n,m)}\bigr]= \left \{ \begin{array}{ll}
\left [0,\frac n 2 +1 \right], & m=0\,,   \\ \\
\left [0, 2^{-m}(n+1) \right], & m>0\,, 
\end{array} \right.
\end{equation} 
and  belongs to $C^{n-1}(\RR)$. Moreover, any  system
\begin{equation}
\bPhi^{(n,m)}=\{ \phi^{(n,m)} (\cdot-2^{-m}\alpha): \alpha \in \ZZ\}\,, 
\end{equation} 
is linearly independent and $L_2(\RR)$-stable, forms a partition of unity,  
is {\em totally positive}, and enjoys the variation diminishing property.

\medskip
For any $n\ge 2$ and $\mu > 1$, the mask sequence $\{\ba^{(n,m)}: m \in \ZZ_+\}$ has  the sequence
\begin{equation}
\ba^{(n)}=\left \{ a_\alpha^{(n)}, 0\le \alpha \le n+1\right\}, \qquad a_\alpha^{(n)}=\frac 1 {2^{n+1}} {\binom{n+1}{\alpha}}\,,
\end{equation}
as fundamental mask \cite{GP08}. Since $\ba^{(n)}$ is the mask of the B-spline of degree $n$ having integer knots on $[0,n+1]$,
the sequence of functions $\{h^{(n,m)}_{k}: k \in \ZZ_+\}$, generated by the cascade algorithm
\begin{equation}  \label{cascalgnm}
h^{(n,m)}_{k+1}=\displaystyle \sum_{\alpha \in \sigma_a^{(n,m)}} \ a_\alpha^{(n,m)}\ 
h^{(n,m+1)}_{k}(\cdot -2^{-(m+1)}\alpha)\,,  \qquad m\in \ZZ_+\,,   
\end{equation}
converges strongly to $\phi^{(n,m)}\in L_2(\RR)$.
The convergence of the cascade algorithm implies that the Fourier transform of $\phi^{(n,m)}$ 
is 
\begin{equation} \label{infprodnm}
{\widehat \phi}^{(n,m)}(\omega)  = \prod_{k=m}^{\infty} \,A^{(n,k)}\left(e^{-i \frac{\omega}{2^{k+1}}}\right)\,, \qquad \omega \in \RR\,,
\end{equation}
where
\begin{equation}
A^{(n,m)}(z)=\sum_{\alpha\in \sigma_a^{(n,m)}} \ a^{(n,m)}_\alpha \ z^\alpha\,. 
\end{equation}
A straightforward computation gives
\begin{equation} \label{fattor_Anm}
  \begin{array}{lcl} 
A^{(n,0)}(z) &=&  \frac 12(1+z)= A^{(0)}(z)\,,  \\ \\
A^{(n,m)}(z) &=& \displaystyle{\frac{1}{2^{ n+1+m^{-\mu}} }  (1+z)^{n-1}
\bigl (z^2+2(2^{1+m^{-\mu}}-1)\ z+1\bigr )}= \\ \\
  &=& \frac{1+z}2  \, A^{(n-1,m)}(z) = A^{(0)}(z) \, A^{(n-1,m)}(z) \,,    \quad  m> 0\,.
\end{array} 
\end{equation}
It is worthwhile to observe that any symbol $A^{(n,m)}(z)$ is a Hurwitz polynomial \cite{GP00}, i.e. it has only zeros with negative real part.
Moreover, the fundamental symbol of $A^{(n,m)}(z)$ is the B-spline symbol
\begin{equation}
A^{(n)}(z)=\sum_{\alpha=0}^{n+1} \, a_\alpha^{(n)}\, z^\alpha = \frac 1 {2^{n+1}} \, (1+z)^{n+1}\,.
\end{equation}

\section{Properties of the nonstationary refinable ripplets $\phi^{(n,m)}$}

In this section we analyze some properties of the refinable ripplets $\phi^{(n,m)}$
that are useful in both geometric modeling and signal processing applications.

First of all, let us denote by $B^{(n,m)}$ the $2^m$-dilates of the B-spline of degree $n$ with knots on $2^{-m}\ZZ$, normalized so that
$\widehat B^{(n,m)}(0)=1$, $n \ge1$, $m \in \ZZ_+$. 
In the Fourier space $B^{(n,m)}$ satisfies the refinable equation 
\begin{equation}
{\widehat B}^{(n,m)}(\omega) = A^{(n)}\left(e^{-i \frac{\omega}{2^{m+1}}}\right)\, {\widehat B}^{(n,m+1)}(\omega)\,.
\end{equation}
Its Fourier transform is given by
\begin{equation} \label{Bn_Fou}
{\widehat B}^{(n,m)}(\omega)=\dprod{k=m}{\infty} \, A^{(n)}\left(e^{-i \frac{\omega}{2^{k+1}}}\right)=
\left ( \frac {1-e^{-i \frac{\omega}{2^{m}}  }}{i\omega 2^{-m}}\right ) ^{n+1}\,.
\end{equation}
The refinable functions  $\phi^{(n,m)}$ are generated by a convolution low involving the B-spline $B^{(n,m)}$. 

\medskip
\begin{theorem} \label{Th_conv}
For $n \ge 3$, $\phi^{(n,m)}$, $m\in \ZZ_+$, satisfies the {\em convolution} property
\begin{equation} \label{convm} \begin{array}{l}
\phi^{(n,0)} = B^{(0,1)}* \phi^{(n-1,0)}\,, \\ \\
\phi^{(n,m)} = B^{(0,m)}* \phi^{(n-1,m)}\,, \qquad m >0\,.
\end{array}
\end{equation}

\end{theorem}

\begin{proof}
From (\ref{infprodnm}), (\ref{fattor_Anm}) and (\ref{Bn_Fou}) it follows
$$\begin{array}{lcl}
{\widehat \phi}^{(n,m)}(\omega)  &=& \displaystyle  \dprod{k=m}{\infty} \, A^{(0)}\left ( e^{-i \frac{\omega}{2^{k+1}}}  \right ) \
\dprod{k=m}{\infty} \, A^{(n-1,k)}(e^{-i \frac{\omega}{2^{k+1}}})= \\ \\
& = & {\widehat B}^{(0,m)} (\omega)\,{\widehat \phi}^{(n-1,m)}(\omega)\,
\end{array}$$
for $m >0$, and 
$
{\widehat \phi}^{(n,0)}(\omega) = {\widehat B}^{(0,1)} (\omega)\,{\widehat \phi}^{(n-1,0)}(\omega)\,
$,
for $m=0$. 
The claim follows by applying the inverse Fourier transform to the relations above.
\end{proof} 

\medskip
\noindent
As a first consequence of the theorem above, we can prove that $\phi^{(n,m)}$ satisfies suitable Strang-Fix conditions.

\medskip
\begin{corollary} \label{CorSF}
For any $n\ge2$ and $m >0$, ${\widehat \varphi}^{(n,m)}$ has a zero of order $n-1$ for $\omega=2^{m+1}\pi \alpha$,
$\alpha \in \ZZ\backslash \{0\}$. ${\widehat \varphi}^{(n,0)}$ has a simple zero for $\omega=2\pi \alpha$,
$\alpha \in \ZZ\backslash \{0\}$.
\end{corollary}

\begin{proof}
By repeated application of Th.~\ref{Th_conv} we get
\begin{equation} \label{conv_TFphi}
\begin{array}{l}
{\widehat \phi}^{(n,m)}(\omega)={\widehat B}^{(n-2,m)} (\omega)\,F^{(m)}(\omega)\,, \\ \\
{\widehat \phi}^{(n,0)}(\omega) = {\widehat B}^{(n-2,1)} (\omega)\,F^{(0)}(\omega)\,,
\end{array}
\end{equation}
where 
$$
F^{(m)}(\omega)  = \dprod{k=m}{\infty} \, A^{(1,k)}(e^{-i \frac{\omega}{2^{k+1}}})\,.
$$
Since the symbols $A^{(1,m)}(z)$ have the hat function symbol $A^{(1)}(z)$ as fundamental symbol, 
the infinite product converges \cite{GL99}. Moreover, for $m >0$ the symbol $A^{(1,m)}(z)$ has no zeros on the unit circle, 
so that $F^{(m)}(\omega)$ has no zeros, too. Thus, ${\widehat \phi}^{(n,m)}(\omega)$, $m>0$, has zeros of the same order 
as ${\widehat B}^{(n-2,m)}$ does, and the claim follows.
\\
For $m=0$, 
$${\widehat \phi}^{(n,0)}(\omega) ={\widehat B}^{(n-2,1)}(\omega)\,A^{(0)}(e^{-i \frac{\omega}{2}}) \, F^{(1)}(\omega)\,. $$$A^{(1,0)}(e^{-i \frac{\omega}{2}})$ has a simple zero for $\omega=2(2\alpha+1)\pi$, $\alpha \in \ZZ$,
while ${\widehat B}^{(n-2,1)}$ has a zero of order $n-1$ for
$\omega=4\alpha\pi$, $\alpha \in \ZZ\backslash\{0\}$. Thus, ${\widehat \phi}^{(n,0)}(\omega)$
has a simple zero for $\omega=2\alpha\pi$, $\alpha \in \ZZ\backslash\{0\}$.
\end{proof}
\medskip

The Strang-Fix conditions allow us to conclude that polynomials of suitable degree are contained in  
the space generated by the $2^{-m}$-translates of $\phi^{(n,m)}$, i.e. the space
\begin{equation}
V^{(n,m)}= \overline{span} \left \{\phi^{(n,m)} (\cdot-2^{-m}\alpha): \alpha \in \ZZ \right\}\,, \qquad m \in \ZZ_+\,. 
\end{equation}

\begin{theorem} \label{th_poly} Let $\Pi_d$ be the space of polynomials up to degree $d$. For any $n\ge 2$ and $m>0$, the space $V^{(n,m)}$  contains $\Pi_{n-2}$,
i.e. for any polynomial $p \in \Pi_{n-2}$ there exists a sequence of real numbers $\{\gamma_\alpha^{(n,m)}: \alpha \in \ZZ\}$ such that
\begin{equation}
p=\sum_{\alpha \in \ZZ} \ \gamma_\alpha^{(n,m)} \ \phi^{(n,m)}(\cdot-2^{-m}\alpha)\,.
\end{equation}
The space $V^{(n,0)}$ contains  the space of the constants $\Pi_0$.
\end{theorem}

\medskip
\noindent
{\em Note}. The roots of the nonstationary symbols (\ref{fattor_Anm}) cannot
fulfill the hypotheses of Th. 1 in \cite{VBU07}. Thus, the refinable functions $\phi^{(n,m)}$ cannot 
generate exponential polynomials. Actually, we are interest in the construction of efficient decomposition 
and reconstruction formulas for general (possibly non exponential) signals. To this end high algebraic polynomial 
generation is more effective since it induces a high number of vanishing moments in the analyzing wavelet.
\medskip

\noindent
From the results above it follows the approximation order of the system $\bPhi^{(n,m)}$  \cite{Ji97}.  

\medskip
\begin{corollary}
For any $n\ge 2$  and $m>0$, the system $\bPhi^{(n,m)}$ has {\em approximation order} $n-1$, i.e. for any $f \in L_2(\RR)$ there exists a constant $C_f$, independent from $m$, such that
\begin{equation}
\inf_{f_m \in V^{(n,m)}} \|f-f_m\|_2 \le C_f \,  2^{-m(n-1)}\,.  
\end{equation}
The system $\bPhi^{(n,0)}$ has approximation order 1.
\end{corollary}

\medskip
\noindent
{\em Note}. Even if the system $\bPhi^{(n,0)}$ reproduces just the constants, polynomials of higher degree can be represented by suitable 
integer translates of $\phi^{(n,0)}(2^{-1}\cdot)$. In fact, $\phi^{(n,0)}(2^{-1}\cdot)\in C^{n-1}(\RR)$ and its Fourier transform 
has a zero of order $n-1$ for $\omega=2\alpha\pi$, $\alpha \in \ZZ\backslash\{0\}$. Thus, polynomials
of degree $n-2$ can be represented by the integer shifts of the dilate $\phi^{(n,0)}(2^{-1}\cdot)$ and the system $\{\phi^{(n,0)}(2^{-1}\cdot-\alpha): \alpha \in \ZZ\}$ has approximation order $n-1$ (cf. \cite{DM93}). 
\bigskip

\noindent
Interestingly enough, the convolution property (\ref{convm}) allows us to prove a {\it differentiation rule} 
for the functions  $\phi^{(n,m)}$.

\medskip
\begin{theorem}  \label{derphinm}
Let $\nabla_h^r$ be the backward finite difference operator defined recursively as
$$
\nabla_h f = \frac 1 h \bigl (f - f(\cdot-h) \bigr)\,, \qquad \nabla_h^0 f =  f\,, \quad 
\nabla_h ^r f = \nabla_h ( \nabla_h^{r-1}) f\,, \quad r\ge 1\,.
$$
For $n\ge2$, the derivatives of $\phi^{(n,m)}$ are given by
\begin{equation} \label{dervm} 
\begin{array}{lll}
D^r \, \phi^{(n,0)}=  \nabla_{2^{-1}}^r \, \phi^{(n-r,0)}, &\quad  r\le n-1\,, \\ \\
D^r \, \phi^{(n,m)}=  \nabla_{2^{-m}}^r \, \phi^{(n-r,m)},&\quad  r\le n-1\,, &\quad m>0\,.
\end{array}
\end{equation}
\end{theorem}

\begin{proof}
We will prove (\ref{dervm}) just for $m >0$; the case $m=0$ can be proved in a similar way. 
\\
From (\ref{Bn_Fou}) and the first of (\ref{conv_TFphi}) it follows
$$  \begin{array}{lcl} { \widehat {\left ( D  \ \phi^{(n,m)} \right ) } }(\omega) &=& i\omega  \ {\widehat \phi^{(n,m)}}(\omega)=
i\omega  \ {\widehat B}^{(n-2,m)} (\omega) \,F^{(m)}(\omega) = \\ \\
&=& \left ( \frac{1-e^{-i2^{-m}\omega }}{2^{-m}}\right ) \ {\widehat B}^{(n-3,m)}(\omega) \ F^{(m)}(\omega) = \\ \\
&=& \left ( \frac{1-e^{-i2^{-m}\omega }}{2^{-m}}\right ) \ {\widehat \phi^{(n-1,m)}} (\omega)\,.
\end{array}$$
The inverse Fourier transform gives
$$
D \, \phi^{(n,m)}=  \nabla_{2^{-m}} \, \phi^{(n-1,m)},
$$
which is the derivation rule for $r=1$. Repeated applications of this rule give the derivation rules of higher order.
\end{proof}
\medskip

From (\ref{dervm}) and the variation diminishing property, we can infer that the derivative 
$D^r \, \varphi^{(n,m)}$, $m\in \ZZ_+\,$, has the same behavior of the finite difference 
\begin{equation} \label{derdiffvm}
\nabla_{h}^r \, \phi^{(n-r,m)} = \displaystyle \frac 1 {h^{r}} \, \sum_{\alpha=0}^r \, (-1)^{\alpha} \, \binom{r}{\alpha} \, 
\phi^{(n-r,m)}(\cdot -h\, \alpha)\,, \qquad m\in \ZZ_+\,,
\end{equation}
with $h=2^{-1}$ for $m=0$ and $h=2^{-m}$ for $m>0$. In particular, 
the number of strict sign changes of $D^r \, \varphi^{(n,m)}$ is not greater than the number of strict sign changes of the sequence 
$$\bc^r=\{c^r_\alpha: 0 \le \alpha \le r\}=\bigl \{ (-1)^{\alpha} \, \binom{r}{\alpha}: 0 \le \alpha \le r\}\,,$$ so that we can infer the shape of $\phi^{(n,m)}$ from $S^-\left(\bc^r\right)$.

\medskip
\begin{theorem}  
For any $n>2$ and $m\in \ZZ_+$, $\phi^{(n,m)}$ is {\em bell-shaped}, i.e. $\phi^{(n,m)}$ is centrally symmetric, strictly increasing on $[0,|supp \ \phi^{(n,m)}|/2]$, 
and its second derivative has just 2 sign changes.
\end{theorem}

\begin{proof}
Since any mask $\ba^{(n,m)}$ is centrally symmetric, any $\phi^{(n,m)}$ is centrally symmetric, too. As a consequence $D\,\varphi^{(n,m)}$ is
centrally antisymmetric and $D^2\,\varphi^{(n,m)}$ is centrally symmetric.
For $m>0$  (\ref{dervm}) and (\ref{derdiffvm}) give 
\begin{equation} \label{der12m}
\begin{array}{l}
D\,\varphi^{(n,m)} = \displaystyle \frac 1{h} \left (\varphi^{(n-1,m)}-\varphi^{(n-1,m)}(\cdot-h) \right )\,, \\ \\
D^2\,\varphi^{(n,m)} =\displaystyle \frac 1{h^2} \left (\varphi^{(n-1,m)}-2\,\varphi^{(n-1,m)}(\cdot-h)+\varphi^{(n-1,m)}(\cdot-2\,h) \right )\,,
\end{array}
\end{equation}
where $h=2^{-m}$, while for $m=0$ (\ref{dervm})-(\ref{derdiffvm}) together with (\ref{REphinm}) yield 
\begin{equation} \label{der120}
\begin{array}{lcl}
D \, \varphi^{(n,0)} &= &\frac 1h \left(\frac12 \, \varphi^{(n-1,1)}(x)-\frac12 \, \varphi^{(n-1,1)}(x-2\,h)\right)\,, \\ \\
D^2 \, \varphi^{(n,0)} &= & \frac 1{h^2} \bigl (\frac12 \, \varphi^{(n-1,1)}(x)-\frac12 \, \varphi^{(n-1,1)}(x-h)- \\ \\
& &  \frac12 \, \varphi^{(n-1,1)}(x-2\,h)+\frac12 \, \varphi^{(n-1,1)}(x-3\,h) \bigr )\,, 
\end{array}
\end{equation}
where $h=2^{-1}$.
Since $\varphi^{(n-1,m)}$ is positive in $(0,|\supp \, \phi^{(n-1,m)}|)$ and vanishes for $x\le 0$, $D\varphi^{(n,m)}$ 
and $D^2\varphi^{(n,m)}$ are positive for $0<x<h$.
Now, by the variation diminishing property from the first relations in (\ref{der12m}) and (\ref{der120}) it follows that $D\varphi^{(n,m)}$ has just one sign change which has to be in $|\supp \, \phi^{(n,m)}|/2$; 
thus, $D\varphi^{(n,m)}>0$ in $(0,|\supp \, \phi^{(n,m)}|/2)$ and $D\varphi^{(n,m)}<0$ in $(|\supp \, \phi^{(n,m)}|/2,|\supp \, \phi^{(n,m)}|)$. 
As for the second derivative, a direct computation shows that $D^2\varphi^{(n,m)}$ is negative in 
$|\supp \, \phi^{(n,m)}|/2$ so that $D^2\varphi^{(n,m)}$ has at least two sign changes; but from the variation diminishing property and the second relations in (\ref{der12m}) and (\ref{der120}) it follows that $D^2\varphi^{(n,m)}$ cannot have more than two sign changes, so proving the claim.
\end{proof}

\section{Nonstationary Prewavelets}

From the results in the previous sections it follows that, for any $n\ge 2$ held fix, the spaces 
$V^{(n,m)}$, $m \in \ZZ_+$, generate a nonstationary multiresolution analysis as defined in Section~2.

We note that the $L_2(\RR)$-stability of the basis $\bPhi^{(n,m)}$ implies that the symbol of the autocorrelation of $\varphi^{(n,m)}$,
i.e. the polynomial
\begin{equation}
\begin{array}{lcl}
\rho_\phi^{(n,m)}(\omega)&=&\displaystyle \sum_{\alpha \in \ZZ} \left ( \int_\RR \, \varphi^{(n,m)} \, \varphi^{(n,m)}(\cdot+2^{-m}\alpha) 
\right) \ e^{-i\omega\,2^{-m}\,\alpha}= \\ \\
& = & 2^m \,\displaystyle \sum_{\alpha \in \ZZ} \left | {\widehat \varphi}^{(n,m)}(\omega+ 2^{m+1}\pi \,\alpha )\right |^2\,,
\end{array}
\end{equation}
is non vanishing for any $\omega \in \RR$. 
The vector ${\mathbf \eta}^{(n,m)}= [\eta^{(n,m)}_{\alpha}]^T$, where 
$
\eta^{(n,m)}_{\alpha} = \int_\RR \, \phi^{(n,m)} \, \phi^{(n,m)}(\cdot+2^{-m}\alpha))\,,
$
 is the eigenvector corresponding to the eigenvalue 1 of the transition operator 
\begin{equation}
(T^{(n,m)} \ \blambda)_\alpha=  2\,\sum_{\beta \in \ZZ} \ \check a_{2\alpha-\beta}^{(n,m)} \ \lambda_\beta\,,  \qquad \alpha \in \ZZ\,,
\qquad \blambda \in \ell_0(\ZZ)\,,
\end{equation}
where
$ \check \ba^{(n,m)}= \{ \check a^{(n,m)}_\alpha=\sum_{\beta \in \sigma_a^{(n,m)}} \, a_\beta ^{(n,m)} \, a_{\beta-\alpha} ^{(n,m)}, \alpha \in \ZZ\}$,  is the autocorrelation of the mask $a^{(n,m)}$\,.
The sequence $\{\eta^{(n,m)}_{\alpha}\}$ is positive, compactly supported with
$$
\sigma_\eta^{(n,m)} = supp \{\eta^{(n,m)}_{\alpha}\} = \left \{ \begin{array}{ll}
\bigl [-\left[ \dfrac n2 +1 \right], \left[ \dfrac n2+1\right]\bigr]\,, & m=0\,, \\ \\
{[-n-1,n+1]}\,, & m>0\,,
\end{array} \right.
$$
and centrally symmetric. As a consequence $\rho^{(n,m)} \in \RR$  with
\begin{equation}
0<\rho_\phi^{(n,m)}(2^m\pi)\le \rho_\phi^{(n,m)}(\omega) \le \rho_\phi^{(n,m)}(0) = 2^m\,.
\end{equation}
Since
\begin{equation}
\begin{array}{lcl}
\rho_\phi^{(n,m)}(2^m\pi)&=&\sum_{\alpha \in \ZZ} \, (-1)^\alpha \, \left ( \int_\RR \, \phi^{(n,m)} \, \phi^{(n,m)}(\cdot+2^{-m}\alpha)) \right)= \\ \\
&=& \sum_{\alpha \in \ZZ} \, (-1)^\alpha \bigl(\phi^{(n,m)}*\phi^{(n,m)})\bigr)(L_{(n,m)}+2^{-m}\alpha)\,,
\end{array}
\end{equation}
at any level $m$ the basis $\bPhi^{(n,m)}$ is a non orthogonal basis.

The nonstationary multiresolution analysis $\{ V^{(n,m)}: m \in \ZZ_+\}$ allows us to define a wavelet space sequence $\{ W^{(n,m)}: m\in \ZZ_+\}$, where each space $W^{(n,m)}$ is the orthogonal complement of  
$V^{(n,m)}$ in $V^{(n,m+1)}$.
Since $\phi^{(n,m)}$ is non orthogonal, orthogonal wavelets with compact support do not exists. On the other hand,
due to the $L_2(\RR)$-stability and the compact support of each $\phi^{(n,m)}$, it is always possible to construct compactly supported semiorthogonal wavelets \cite[Th. 3.12]{BDR93}.

The explicit expression of the prewavelet of minimal support can be obtained generalizing to the
nonstationary case the results in \cite{Mi91} (see also \cite{GP08-2}). 

\medskip
\begin{theorem} \label{th_prewave}
For any $ n \ge 2$, the functions 
\begin{equation} \label{prewave} 
\begin{array}{l}
\displaystyle \psi^{(n,0)} = \sum_{\alpha = -n}^{n+1} \, (-1)^\alpha \, g_{\alpha-1}^{(n,0)} \ 
\phi^{(n,1)}(\cdot-2^{-1}\alpha)\,, \\ \\
\displaystyle \psi^{(n,m)} = \sum_{\alpha =-2n}^{n+1} \, (-1)^\alpha \, g_{\alpha-1}^{(n,m)} \ 
\phi^{(n,m+1)}(\cdot-2^{-(m+1)}\alpha)\,, \qquad m>0\,,
\end{array}
\end{equation}
where
\begin{equation} \label{munm}
g_{\alpha}^{(n,m)} =   \int_{\RR} \phi^{(n,m)}\
\phi^{(n,m+1)}(\cdot+2^{-(m+1)}\alpha)\,, \qquad \alpha \in \ZZ\,,
\end{equation} 
are the prewavelets generating the wavelet spaces 
\begin{equation}
W^{(n,m)}= \overline{span}\left \{\psi^{(n,m)} (\cdot-2^{-m}\alpha): \alpha \in \ZZ \right\}\,, \qquad m \in \ZZ_+\,. 
\end{equation}
The functions $\psi^{(n,m)}$ are compactly supported. 
Moreover, any system
\begin{equation}
\bPsi^{(n,m)}=\{\psi^{(n,m)}(x-2^{-m}\alpha): \alpha\in \ZZ\}
\end{equation} 
is $L_2(\RR)$-stable and linearly independent. 
Finally, $\psi^{(n,m)}$ has $n-1$ vanishing moments.  
\end{theorem}

\begin{proof}
A straightforward computation yields 
$$\begin{array}{ll}
\displaystyle \int_\RR & \psi^{(n,m)}) \ \phi^{(n,m)}(\cdot-2^{-m}\beta)  =\\ \\
&= \displaystyle \sum_\alpha \, (-1)^\alpha \, g_{\alpha-1}^{(n,m)} \ \int_\RR \ \phi^{(n,m+1)}(\cdot-2^{-(m+1)}\alpha) \ \phi^{(n,m)}(\cdot-2^{-m}\beta)  = \\ \\
&= \displaystyle \sum_\alpha \, (-1)^\alpha \, g_{\alpha-1}^{(n,m)} \, g_{2\beta-\alpha}^{(n,m)}=0\,,
\end{array}
$$
for any $\beta \in \ZZ$\,, i.e. $\psi^{(n,m)}$ is orthogonal to the space $V^{(n,m)}$. Due to the support of $\phi^{(n,m)}$  it follows
$$
\sigma^{(n,m)}_g= supp \, \{g_{\alpha}^{(n,m)}\}= \left \{ \begin{array}{ll}
\left [-n-1,n  \right], & m=0\,,  \\ \\
\left [-2n-1, n \right], & m>0\,, 
\end{array} \right.
$$
so that $\psi^{(n,m)}$ is compactly supported  with
\begin{equation} \label{supp_psi}
supp \ \psi^{(n,m)}= \left \{ \begin{array}{ll}
\left [-\frac n 2, n +1 \right], & m=0, \\ \\
\left [-2^{-m}n, 2^{-m}(n+1) \right], & m>0. 
\end{array} \right.  
\end{equation}
Now, the wavelet equations (\ref{prewave}) in the Fourier space reads
$$
\widehat \psi^{(n,m)}(\omega) = d^{(n,m)}(e^{-i\frac{\omega}{2^{m+1}}}) \, \widehat \phi^{(n,m+1)}(\omega)\,, \qquad m \in \ZZ_+\,,
$$
where
$
d^{(n,m)}(z) = \sum_\alpha \, d_\alpha^{(n,m)} \, z^\alpha$, $d_\alpha^{(n,m)} = (-1)^\alpha \, g_{\alpha-1}^{(n,m)}.
$
An explicit calculation gives
\begin{equation} \label{dnm}
d^{(n,m)}(e^{-i\frac{\omega}{2^{m+1}}}) = - e^{-i\frac{\omega}{2^{m+1}}}
A^{(n,m)}\bigl(- e^{i\frac{\omega}{2^{m+1}}}  \bigr) \rho_\phi^{(n,m+1)}(\omega+2^{m+1}\pi)\,.  
\end{equation}
As a consequence,
$$\begin{array}{lcl}
\rho_\psi^{(n,m)}(\omega)&=&\sum_{\alpha \in \ZZ} \left ( \int_\RR \, \psi^{(n,m)} \, \psi^{(n,m)}(\cdot+2^{-m}\alpha) 
\right) \ e^{-i\omega\,2^{-m}\,\alpha}= \\ \\
&=&2^m \, \sum_{\alpha \in \ZZ} \left |{\widehat \psi} ^{(n,m)}(\omega+2^{m+1}\pi \alpha)\right|^2 =\\ \\
&=&   2^m\,  \sum_{\alpha \in \ZZ} \, |A^{(n,m)}\bigl(- e^{i(\frac{\omega}{2^{m+1}}+\pi \alpha)}  \bigr)|^2 \,
\left |  \rho_\phi^{(n,m+1)}(\omega+2^{m+1}\pi (\alpha+1))  \right|^2 \times \\ \\
& & \times\,  
\left |  {\widehat \phi} ^{(n,m+1)}(\omega+2^{m+1}\pi \alpha)  \right|^2 \,, 
\qquad \omega \in \RR\,,
\end{array}
$$
is non-vanishing.
It follows that at any level $m,$ the system $\bPsi^{(n,m)}$, generating the wavelet space $W^{(n,m)}$,
is linearly independent and $L_2(\RR)$-stable. 
Moreover, $\psi^{(n,m)}$ is non orthogonal with
$$
\int_\RR \, \psi^{(n,m)} \ \psi^{(n,m)}(\cdot+2^{-m}\alpha) = \sum_{\beta} \, (-1)^\beta \eta^{(n,m+1)}_{2\alpha-\beta} \, \check{g}^{(n,m)}_\beta\,, \qquad  \alpha \in \ZZ\,,  
$$
where
$\{\check{g}^{(n,m)}_\alpha\}=\{\sum_\gamma g_\gamma^{(n,m)} \,g_{\gamma-\alpha}^{(n,m)}\}$ is the autocorrelation of 
the sequence $\{g_\alpha^{(n,m)}\}$.
\\
Finally, since $W^{(n,m)}\perp V^{(n,m)}$, from Th. \ref{th_poly} it follows that 
$\int_\RR \, x^d\, \psi^{(n,m)}(x) = 0$, $0\le d \le n-2$, so concluding the proof.
\end{proof}

\medskip
\begin{theorem} \label{Th_suppmin}
For any $n \ge 2$ and $m\in \ZZ_+$, $\psi^{(n,m)}$ is the unique minimally supported wavelet   generating the 
wavelet space $ W^{(n,m)}$.
\end{theorem}

\begin{proof}
The orthogonality conditions
$
\displaystyle \int_\RR \, \psi^{(n,m)}(x) \ \phi^{(n,m)}(x-2^{-m}\beta)  \, dx = 0$, $\beta \in \ZZ,
$
can be proved to be equivalent to the conditions
\begin{equation} \label{dg_orth}
\sum_\alpha \, d^{(n,m)}_\alpha \, g^{(n,m)}_{2\beta-\alpha}=0\,,\qquad \beta \in \ZZ\,.
\end{equation}
Let
$
d_e^{(n,m)}(z) = \sum_\alpha \, d_{2\alpha}^{(n,m)}\, z^\alpha$, $d_o^{(n,m)}(z) = \sum_\alpha \, d_{2\alpha+1}^{(n,m)}\, z^\alpha
$,
and, similarly, for the polynomial $g^{(n,m)}(z)=\sum_\alpha \, g^{(n,m)}_\alpha\, z^\alpha$, let
$\displaystyle g_e^{(n,m)}(z) = \sum_\alpha \, g_{2\alpha}^{(n,m)}\, z^\alpha$, $ g_o^{(n,m)}(z) = \sum_\alpha \, g_{2\alpha+1}^{(n,m)}\, z^\alpha$,
so that
$d^{(n,m)}(z)=d_e^{(n,m)}(z^2)+z\,d_o^{(n,m)}(z^2)$, and $g^{(n,m)}(z)=g_e^{(n,m)}(z^2)+z\,g_o^{(n,m)}(z^2)$.
It follows that conditions (\ref{dg_orth}) are equivalent to
\begin{equation} \label{orth_prew}
d_e^{(n,m)}(z)\,g_e^{(n,m)}(z)+z\,d_o^{(n,m)}(z)\,g_o^{(n,m)}(z)=0\,.
\end{equation}
Thus, $\psi^{(n,m)}$ is the minimally supported prewavelet if and only if $d^{(n,m)}$ is the minimally supported polynomial
satisfying (\ref{orth_prew}), i.e. if and only if $d_e^{(n,m)}(z)$ and $d_o^{(n,m)}(z)$ have no common zeros or, equivalently,
$d^{(n,m)}(z)$ and $d^{(n,m)}(-z)$ have no common zeros for $z=e^{i\frac{\omega}{2^{m+1}}}$.
But this easily follows from (\ref{dnm}) since  any symbol $A^{(n,m)}(z)$ has only zeros with negative real part and $\rho_\phi^{(n,m+1)}(\omega)$ is positive.
\end{proof}
\medskip

We note that $\psi^{(n,m)}$ is the $2^m$-dilate of the prewavelet constructed in \cite{GP08} 
and Th. \ref{th_prewave} and Th. \ref{Th_suppmin} generalize to the nonstationary case the stationary prewavelet constructed in \cite{Mi91}.

Even if the prewavelet coefficients $\{g^{(n,m)}_\alpha\}$ do not have an explicit expression, they can be evaluated 
efficiently by an iterative algorithm.

\medskip
\begin{theorem} \label{Th_4.3}
Let $M^{(n,m)}=[g_{\alpha}^{(n,m)},\alpha \in \sigma_g^{(n,m)}]^T$ be the nonstationary prewavelet coefficients, and let $M^{(n)}=[g_{\alpha}^{(n)},\alpha \in \sigma_g^{(n,m)}]^T$ be the prewavelet coefficients corresponding to the
stationary fundamental mask $\ba^{(n)}$. \\
For any $n\ge2$ and $m\in \ZZ_+$ held fixed, consider the iterative procedure
\begin{equation} \label{algprew}
\left \{ \begin{array}{l}
P^{(n,m)}_0 = M^{(n)} \,,\\
P^{(n,m)}_{k+1} = C^{(n,m)} \ P^{(n,m+1)}_k\,, \qquad k \ge 0\,,
\end{array} \right.
\end{equation}
where $C^{(n,m)}=(c_{2\alpha-\beta}^{(n,m)})$ with
\begin{equation}\label{maskprew}
c^{(n,m)}_\alpha = \displaystyle \sum_{\beta \in \sigma_a^{(n,m)}}\ a^{(n,m)}_\beta \ a^{(n,m+1)}_{\alpha +2\beta}\,, \qquad
  \beta \in \sigma_c^{(n,m)}\,, 
\end{equation}
and
\begin{equation}
\sigma_c^{(n,m)} = \left \{
\begin{array}{ll}
\left [ -2,n+1\right]\,, & m=0\,, \\
\left [ -2(n+1),n+1\right]\,, & m>0\,.
\end{array} \right.  
\end{equation}
The sequence $\{P^{(n,m)}_k\}$ converges strongly to $M^{(n,m)}$ when $k\to \infty$.
Moreover, the following error estimate holds
\begin{equation}
\|P^{(n,m)}_k-M^{(n,m)}\| \le \gamma_{n,m} \  \| M^{(n)}-M^{(n,m+k)}\|\,,
\end{equation}
where $\gamma_{n,m}$ is a positive constant independent from $k$.
\end{theorem}

\begin{proof}
Using the refinable equation (\ref{REphinm}) in (\ref{munm}) we get
$
g_\alpha^{(n,m)} = \sum_\beta \, c_{2\alpha-\beta} \, g_\beta^{(n,m)}
$,
which in matrix form can be written as
$
M^{(n,m)}=C^{(n,m)} \, M^{(n,m+1)}
$.
Repeated applications of both algorithm (\ref{algprew}) and the relation above give
$$
\begin{array}{lcl}
\|P^{(n,m)}_{k+1}-M^{(n,m)}\| &=&  \| \prod_{l=0}^k C^{(n,m+l)} (M^{(n)}-M^{(n,m+k+1)}) \| \le \\ \\
&\le&\| \prod_{l=0}^k \ C^{(n,m+l)}\| \cdot \| M^{(n)}-M^{(n,m+k+1)}\|\,.
\end{array}
$$
Now, let $C^{(n)}=(c_{2\alpha-\beta}^{(n)})$, where
$c^{(n)}_\alpha = \sum\, a^{(n)}_\beta \, a^{(n)}_{\alpha+2\beta}$. We note that $\rho(C^{(n)})=1$.
Since $\ba^{(n)}$ is the fundamental mask of the mask sequence $\{\ba^{(n,m)}\}$, it follows
$
\sum_{m \in \ZZ_+} \, \|C^{(n,m)}-C^{(n)}\| < \infty
$,
and 
$$
\| \prod_{l=m}^{m+k}  C^{(n,l)}\| \le \exp \left ( \sum_{m\in \ZZ_+} \|C^{(n,m)}-C^{(n)}\| \right )<\infty\,,
$$
(cf. \cite[Prop. 2.1]{GL99}). 
Moreover, 
$
\lim_{k \to \infty} M^{(n,k)}= M^{(n)} \,,
$
thus,
$$
\lim_{k \to \infty} \|P^{(n,m)}_{k+1}-M^{(n,m)}\|=0\,,
$$
and the claim follows with $\gamma_{n,m}= \exp \left ( \sum_{m\in \ZZ_+} \|C^{(n,m)}-C^{(n)}\| \right )$.
\end{proof}

\section{Nonstationary Biorthogonal Bases}

The refinable functions $\phi^{(n,m)}$ and the prewavelets $\psi^{(n,m)}$ are linearly independent and $L_2(\RR)$-stable, but they are not orthogonal. As a consequence, the dual bases of $\phi^{(n,m)}$ and $\psi^{(n,m)}$ in $V^{(n,m)}$ and  $W^{(n,m)}$, respectively, have infinite support. 
Nevertheless, compactly supported bases giving rise to efficient
reconstruction and decomposition formula of a given discrete signal, can be obtained introducing biorthogonal bases. 

The theory
of biorthogonal bases in the stationary case \cite{CDF92} can be generalized to the nonstationary framework
(cf. \cite{Ha10,HZ08,VBU07}).
For any multiresolution analysis $\{V^{(n,m)}: m\in \ZZ_+\}$ we can introduce a biorthogonal multiresolution analysis  
$\{\tilde V^{(n,m)}: m\in \ZZ_+\}$ with
\begin{equation} \label{biorthMRA}
\tilde V^{(n,m)} \subset \tilde V^{(n,m+1)}\,, \qquad m \in \ZZ_+\,,
\end{equation}
and biorthogonal wavelet spaces $\{W^{(n,m)}: m\in \ZZ_+\}$ and $\{\tilde W^{(n,m)}: m\in \ZZ_+\}$, such that for any $m\in \ZZ_+$
\begin{equation} \label{biorthwave}
\begin{array}{l}
W^{(n,m)} =  V^{(n,m+1)} \ominus V^{(n,m)}\,, \quad \tilde W^{(n,m)} =  \tilde V^{(n,m+1)} \ominus \tilde V^{(n,m)}\,,  \\ \\ 
W^{(n,m)} \perp \tilde V^{(n,m)}\,, \quad \tilde W^{(n,m)} \perp  V^{(n,m)}\,. 
\end{array}
\end{equation}
At each level $m$, the spaces $V^{(n,m)}$, $\tilde V^{(n,m)}$, $W^{(n,m)}$ and $\tilde W^{(n,m)}$  are generated
by the $2^{-m}$-integer translates of the biorthogonal functions $\phi^{(n,m)}$, $\tilde \phi^{(n,m)}$, $\psi^{(n,m)}$ and $\tilde \psi^{(n,m)}$, respectively, satisfying the following biorthogonality conditions:
\begin{equation} \label{biorthnm}
\begin{array}{l}
\bigl\langle \phi^{(n,m)}(\cdot-2^{-m}\alpha),\tilde \phi^{(n,m)}(\cdot-2^{-m}\beta)\bigr \rangle=\delta_{\alpha\beta}\,, \\ \\
\bigl \langle \psi^{(n,m)}(\cdot-2^{-m}\alpha),\tilde \psi^{(n,m)}(\cdot-2^{-m}\beta)\bigr \rangle=\delta_{\alpha\beta}\,, \\ \\
\bigl \langle \phi^{(n,m)}(\cdot-2^{-m}\alpha),\tilde \psi^{(n,m)}(\cdot-2^{-m}\beta)\bigr \rangle=0\,, \\ \\
\bigl \langle \psi^{(n,m)}(\cdot-2^{-m}\alpha),\tilde \phi^{(n,m)}(\cdot-2^{-m}\beta)\bigr \rangle=0\,.
\end{array}
\end{equation}
We stress that all the biorthogonal functions $\phi^{(n,m)}$, $\tilde \phi^{(n,m)}$, $\psi^{(n,m)}$ and $\tilde \psi^{(n,m)}$, $m \in \ZZ_+$, cannot be obtained each other by dilation. Thus, none of the biorthogonal spaces at level $m$ is a 
scaled versions of the spaces at level $0$.

Biorthogonal bases for the exponential splines were constructed in \cite{VBU07}. Here, we want to construct
the biorthogonal bases associated with the nonstationary refinable functions $\phi^{(n,m)}$, $m \in \ZZ_+$.

As a consequence of (\ref{biorthMRA}) and (\ref{biorthwave}), the wavelet $\psi^{(n,m)}$ belongs to $V^{(n,m+1)}$, while
the biorthogonal functions $\tilde \psi^{(n,m)}$ and $\tilde \phi^{(n,m)}$ belong to  $\tilde V^{(n,m+1)}$, 
so that
\begin{equation}
\begin{array}{l}
\psi^{(n,m)} = \,\sum_{\alpha \in \ZZ} \, q^{(n,m)}_\alpha \, \phi^{(n,m+1)}(\cdot-2^{-(m+1)}\alpha)\,,
\qquad m\in \ZZ_+\,, \\ \\
\tilde \phi^{(n,m)} = \,\sum_{\alpha \in \ZZ} \, \tilde a^{(n,m)}_\alpha \, \tilde \phi^{(n,m+1)}(\cdot-2^{-(m+1)}\alpha)\,,
\qquad m\in \ZZ_+\,, \\ \\
\tilde \psi^{(n,m)} = \,\sum_{\alpha \in \ZZ} \, \tilde q^{(n,m)}_\alpha \, \tilde \phi^{(n,m+1)}(\cdot-2^{-(m+1)}\alpha)\,,
\qquad m\in \ZZ_+\,.
\end{array}
\end{equation}
Moreover, perfect reconstruction at each level $m$ is guaranteed if for any $m\in \ZZ_+$ the biorthogonal symbols 
\begin{equation}
\begin{array}{l}
 A^{(n,m)}(z)= \sum_{\alpha \in \ZZ} \, a^{(n,m)}_\alpha \, z^\alpha\,, \quad \tilde A^{(n,m)}(z)= \sum_{\alpha \in \ZZ} \, \tilde a^{(n,m)}_\alpha \, z^\alpha\,, \\ \\
Q^{(n,m)}(z)= \sum_{\alpha \in \ZZ} \, q^{(n,m)}_\alpha \, z^\alpha\,, \quad \tilde Q^{(n,m)}(z)= \sum_{\alpha \in \ZZ} \, \tilde q^{(n,m)}_\alpha \, z^\alpha\,,
\end{array}
\end{equation}
satisfy 
\begin{equation} \label{eqBezout}
A^{(n,m)}(z) \, \tilde A^{(n,m)}(z^{-1})+A^{(n,m)}(-z) \, \tilde A^{(n,m)}(-z^{-1})=1\,,    %(2.8) di [CDF92]
\end{equation}
with
\begin{equation}
 Q^{(n,m)}(z) = - \tilde A ^{(n,m)}(-z^{-1})\,, \qquad 
\tilde Q^{(n,m)}(z) = A^{(n,m)}(-z^{-1})\,.   %(2.6) di [CDF92]
\end{equation}
Identity (\ref{eqBezout}) is a Bezout's equation which has a unique polynomial solution $\tilde A^{(n,m)}$ of a given degree \cite{Dau92}.

Under mild conditions on the symbols $A^{(n,m)}$ and $\tilde A^{(n,m)}$, biorthogonality conditions (\ref{biorthnm}) guarantee that the biorthogonal bases
\begin{equation}
\begin{array}{l}
\bPhi^{(n,m)}=\{\phi^{(n,m)}(\cdot-2^{-m}\alpha), \alpha \in \ZZ\}\,, \quad
\bPsi^{(n,m)}=\{\psi^{(n,m)}(\cdot-2^{-m}\alpha), \alpha \in \ZZ\}\,, \\
\tilde \bPhi^{(n,m)}=\{\tilde \phi^{(n,m)}(\cdot-2^{-m}\alpha), \alpha \in \ZZ\}\,, \quad
\tilde \bPsi^{(n,m)}=\{\tilde \psi^{(n,m)}(\cdot-2^{-m}\alpha), \alpha \in \ZZ\}\,,
\end{array}
\end{equation}
are $L_2(\RR)$-stable, so that, for any $f \in L_2(\RR)$, the following decomposition formula holds
\begin{equation}
f = f_{m_0} + \displaystyle \sum_{m\ge m_0} \sum_ {\alpha\in \ZZ} \, \bigl \langle f,\tilde \psi^{(n,m)}(\cdot-2^{-m}\alpha)\bigr\rangle \, \psi^{(n,m)}(\cdot-2^{-m}\alpha)\,,
\end{equation} 
where 
\begin{equation}
f_{m_0}=\sum_ {\alpha\in \ZZ} \, \bigl \langle f,\tilde \phi^{(n,m_0)}(\cdot-2^{-m_0}\alpha)\bigr\rangle \, \phi^{(n,m_0)}(\cdot-2^{-m_0}\alpha)
\end{equation}
is the ${m_0}$-level approximation. Some examples of nonstationary biorthogonal bases will be given in the next section.

The sequences $\{a_\alpha\}$, $\{\tilde a_\alpha\}$, $\{q_\alpha\}$, $\{\tilde q_\alpha\}$  are pairs of biorthogonal FIR filters
that give rise to the decomposition and reconstruction algorithms
\begin{equation} \label{dec_alg}
\lambda_\alpha^m=\frac 1 {\sqrt{2}} \sum_{\beta\in \ZZ} \ {\tilde a}_{\beta-2\alpha}^{(n,m)} \ \lambda_\beta^{m+1}\,,
\qquad \zeta_\alpha^m=\frac 1 {\sqrt{2}} \sum_{\beta\in \ZZ} \ {\tilde q}_{\beta-2\alpha}^{(n,m)} \ \lambda_\beta^{m+1}\,,
\end{equation}
\begin{equation}\label{conv2}
\lambda_\alpha^{m+1}= \frac 1 {\sqrt{2}} \left [ \, \sum_{\beta\in \ZZ} \ a_{\alpha-2\beta}^{(n,m)}\ \lambda_\beta^m+\sum_{\beta\in\ZZ} \, q_{\alpha-2\beta}^{(n,m)} \ \zeta_\beta^m\right ] \,,   
\end{equation}
which can be efficiently used for the analysis and synthesis of a given data sequence $\blambda^0=\{\lambda^0_\alpha: \alpha \in \ZZ\}$. 

\section{A case study}

In this section we give some examples of both nonstationary prewavelets and  biorthogonal bases in the 
case when $n=3$.
In this case the nonstationary refinable functions $\varphi^{(3,m)}$, $m\ge0$, belong to $C^2(\RR)$, i.e.
they have the same smoothness as the cubic B-spline. Interestingly enough, any $\varphi^{(3,m)}$  
with $m> 0$ has the same support as $B^{(3,m)}$, i.e. $[0,4\cdot2^{-m}]$, while $\varphi^{(3,0)}$ is more localized 
in the scale-time plane having $supp \,\varphi^{(3,0)}=[0,5/2]$, a property that appears very useful in applications
(see the example below). In order to obtain refinable beses and nonstationary filters significantly different from 
those ones generated by the cubic B-spline, we choose $\mu=1.1$ as a value for the tension parameter. 
\\
\begin{table} 
\caption{Numerical values (rounded to the forth digit) of the mask coefficients $a^{(3,m)}_0$, $a^{(3,m)}_1$, $a^{(3,m)}_2$ for $m=0,\ldots,8$. Here $\mu=1.1$}
\label{Tab_1}
{\small 
\begin{tabular}{|c|ccccccccc|}
\hline
$m$ & 0 & 1 & 2 & 3 & 4 & 5 & 6 & 7 & 8 \\
\hline
$a^{(3,m)}_0$ & 0.5  &  0.0313  &  0.0452  &  0.0508  &  0.0537  &  0.0555  &  0.0567  &  0.0576  &  0.0583 \\
$a^{(3,m)}_1$ & 0.5  &  0.2500  &  0.2500  &  0.2500  &  0.2500  &  0.2500  &  0.2500  &  0.2500  &  0.2500 \\
$a^{(3,m)}_2$ &  0    &  0.4375  &  0.4095  &  0.3984  &  0.3925  &  0.3889  &  0.3865  &  0.3848  &  0.3835 \\
\hline
\end{tabular}
}

\end{table}
The coefficients of the mask $\mathbf a^{(3,0)}=\{ a^{(3,0)}_0,a^{(3,0)}_1\}$ are
$a^{(3,0)}_0=a^{(3,0)}_1=\frac 1 2$, while the coefficients of the mask 
$\mathbf a^{(3,m)}=\{ a^{(3,m)}_0,a^{(3,m)}_1,a^{(3,m)}_2,a^{(3,m)}_3,a^{(3,m)}_4\}$ for $m>0$ are
$$
a^{(3,m)}_0 =a^{(3,m)}_4 =\displaystyle 2^{-4-m^{-\mu}}\,, \qquad 
a^{(3,m)}_1 =a^{(3,m)}_3 = \displaystyle \frac 1 4\,, \quad
a^{(3,m)}_2 = \displaystyle \frac 1 2 - 2^{-3-m^{-\mu}}\,.
$$
The numerical values (rounded to the forth digit) of the mask coefficients are listed in Tab.~\ref{Tab_1}, while their behavior  
is shown in Fig.~\ref{Fig_1} (right). The behavior of $\varphi^{(3,0)}$ in comparison with $B^{(3,0)}$ is displayed in 
Fig.~\ref{Fig_1} (left). 
\begin{figure}
\centering
\begin{tabular}{ccc}
\includegraphics[width=5cm]{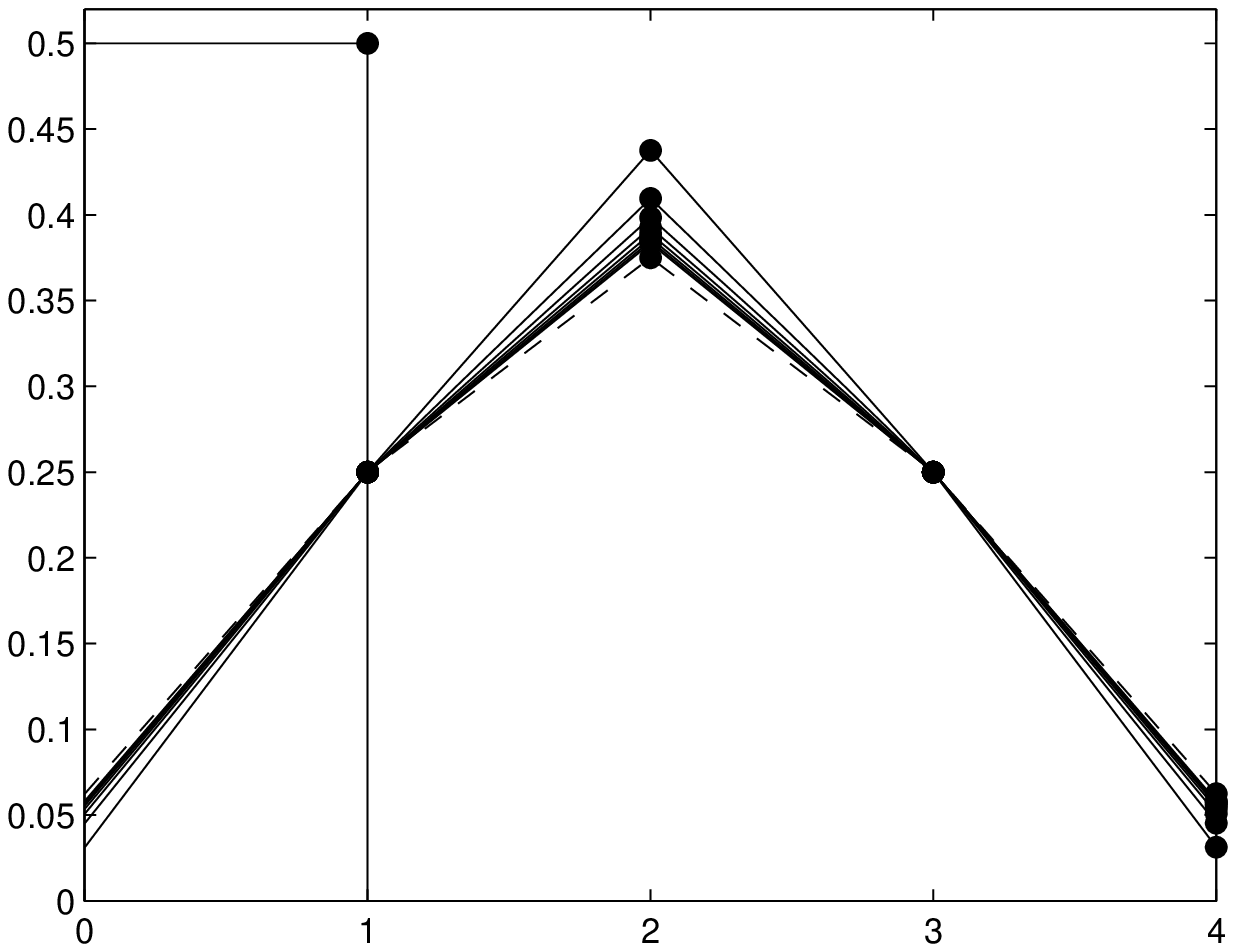}
& \quad &
\includegraphics[width=5cm]{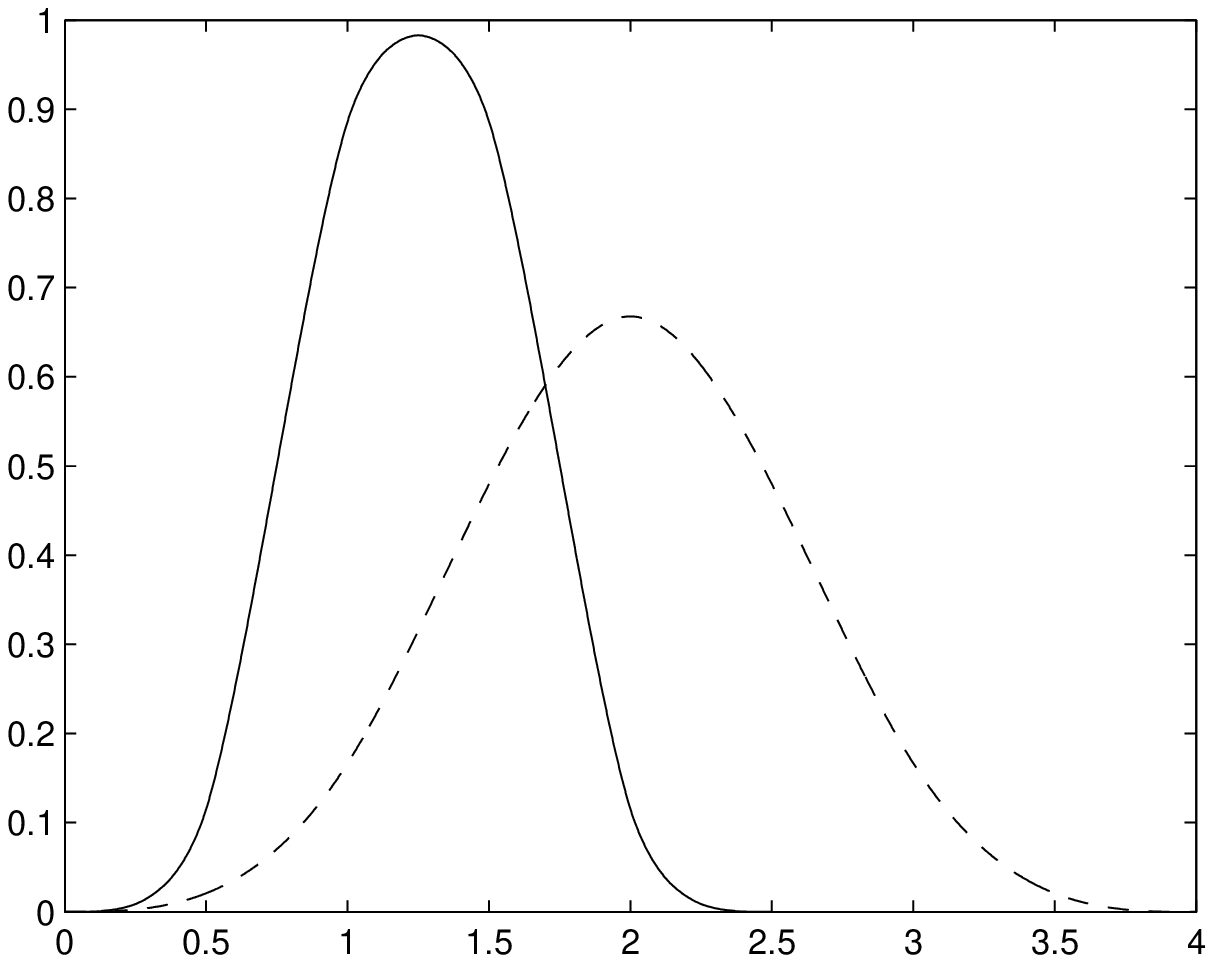}
\end{tabular}
\caption{The nonstationary mask coefficients listed in Tab.~\ref{Tab_1} (left) and $\varphi^{(3,0)}$ (right). 
The stationary mask of the cubic B-spline and the cubic B-spline itself are also displayed (dashed line)}
\label{Fig_1}
\end{figure}
The nonstationary prewavelets $\psi^{(3,m)}$ are given by:
$$ 
\begin{array}{l}
\displaystyle \psi^{(3,0)} = \,\sum_{\alpha = -3}^{4} \, (-1)^\alpha \, g_{\alpha-1}^{(3,0)} \ 
\phi^{(3,1)}(\cdot-2^{-1}\alpha)\,, \\ \\
\displaystyle \psi^{(3,m)} = \,\sum_{\alpha =-6}^{4} \, (-1)^\alpha \, g_{\alpha-1}^{(3,m)} \ 
\phi^{(3,m+1)}(\cdot-2^{-(m+1)}\alpha)\,,  \qquad m>0\,,
\end{array}
$$
where the prewavelet coefficients $\{g_{\alpha}^{(3,n)}\}$ can be evaluated by the algorithm in Th.~\ref{Th_4.3}. 
From (\ref{supp_psi}) it follows that $supp \, \psi^{(3,0)} =[-3/2,4]$, while for $m>0$ $supp \, \psi^{(3,m)} =[-2^{-m}\,3,2^{-m}\,4]$.
We notice that $\psi^{(3,0)}$ is more localized in the scale-time plane than both $\psi^{(3,m)}$, $m>0$, and the B-spline prewavelet.
The prewavelet mask coefficients, rounded to the forth digit, are
$g^{(3,0)}_{-4}=-g^{(3,0)}_3=-0.0015$,  $g^{(3,0)}_{-3}=-g^{(3,0)}_2=0.0259$,
$g^{(3,0)}_{-2}=-g^{(3,0)}_1=-0.1479$,  $g^{(3,0)}_{-1}=-g^{(3,0)}_0=0.3244$.
\\
In Fig.~\ref{Fig_2} the behavior of $\{g^{(3,0)}_\alpha\}$ and $\psi^{(3,0)}$ are displayed.
\begin{figure}
\centering
\begin{tabular}{ccc}
\includegraphics[width=5cm]{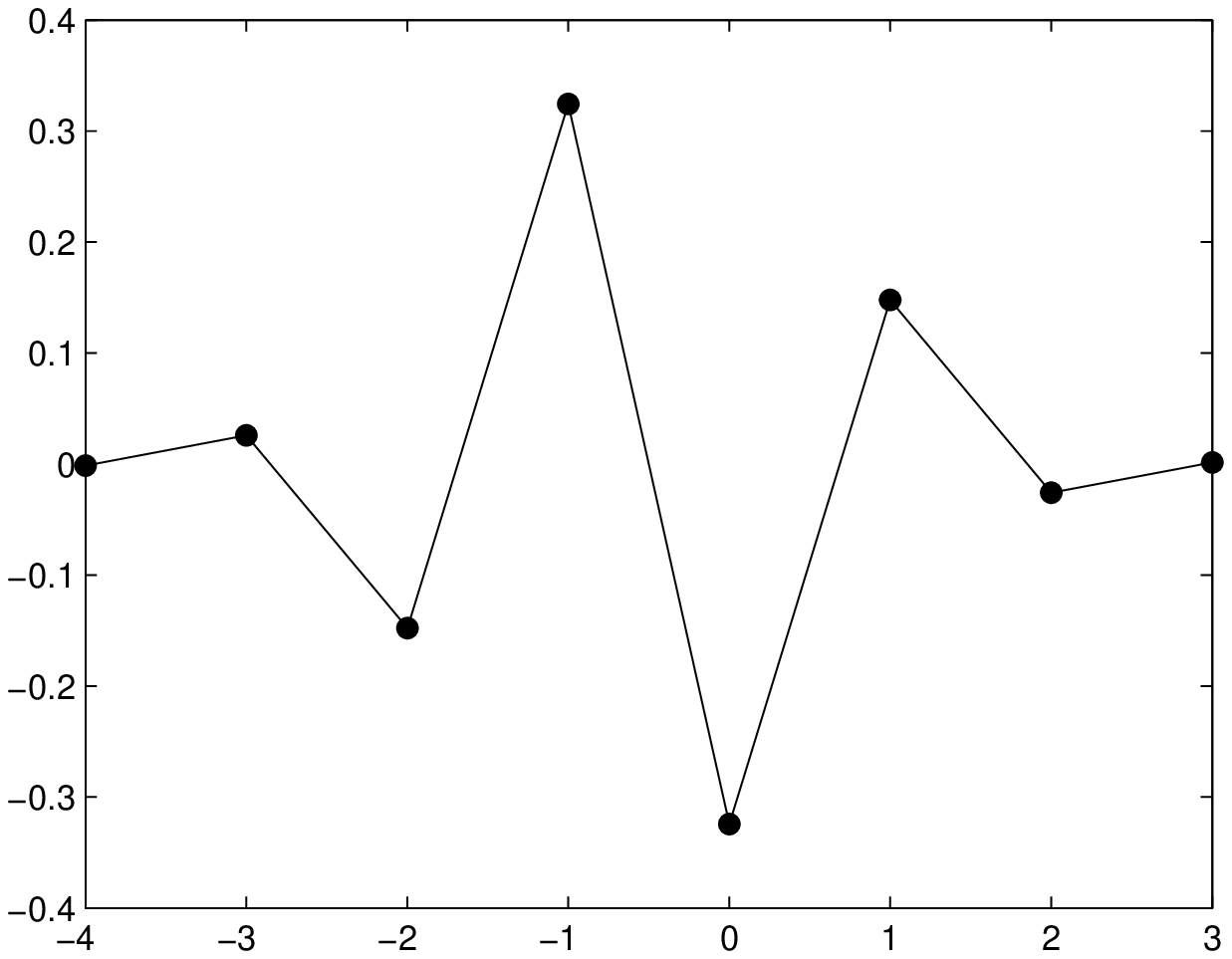}
& \quad &
\includegraphics[width=5cm]{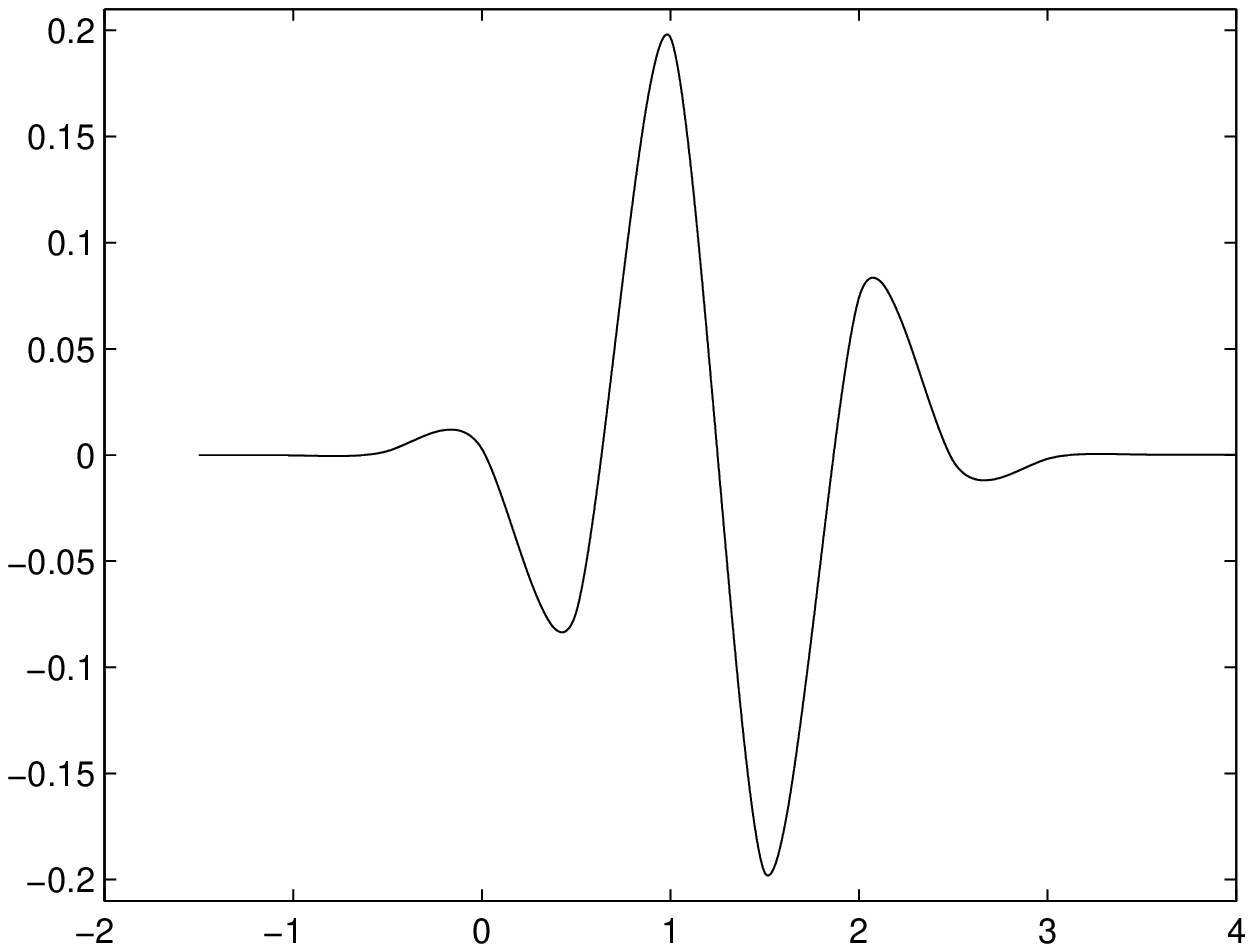}
\end{tabular}
\caption{The prewavelet mask $\{g^{(3,0)}\}$ (left) and $\psi^{(3,0)}$ (right) }
\label{Fig_2}
\end{figure}
Finally, we give the explicit expression of the biorthogonal masks $\widetilde {\bf a}^{(3,m)}$, $m\ge 0$.
It is well known that the biorthogonal mask of $\mathbf a^{(3,0)}$ is
$
\widetilde {\mathbf a}^{(3,0)} = \{\widetilde a^{(3,0)}_0, \widetilde a^{(3,0)}_1 \} = \bigl \{ \frac 1 2 , -\frac 1 2 \bigr \}.
$
\\
In order to fulfill conditions ensuring the existence of the biorthogonal refinable function $\widetilde \varphi^{(3,m)}$, for $m> 0$ 
we construct the biorthogonal mask $\widetilde {\bf a}^{(3,m)}$ with support $[0,14]$. Its explicit expression is given by

{\small \begin{tabular}{l}
$\widetilde {a}^{(3,m)}_0 = \widetilde a^{(3,m)}_{14} = \displaystyle \frac{8^{-3 - h}}{-4+2^h}(128 + 2^{6 + h} + 5\cdot4^{1 + h} + 5\cdot8^h)$ \\ \\
$\widetilde {a}^{(3,m)}_1 = \widetilde a^{(3,m)}_{13} = -\displaystyle \frac{4^{-5 - h}}{-4+2^h}(128 + 2^{6 + h} + 5\cdot4^{1 + h} + 5\cdot8^h)$ \\ \\
$\widetilde {a}^{(3,m)}_2 = \widetilde a^{(3,m)}_{12} = -\displaystyle \frac{8^{-3 - h}}{-4+2^h}(640 + 7\cdot2^{6 + h} + 33\cdot4^{1 + h} + 29\cdot8^h - 5\cdot16^h)$ \\ \\
$\widetilde {a}^{(3,m)}_3 = \widetilde a^{(3,m)}_{11} = \displaystyle \frac{2^{-9 - 2\cdot h}}{-4+2^h}(128 + 3\cdot2^{6 + h} + 17\cdot4^{1 + h} + 17\cdot8^h)$ \\ \\
$\widetilde {a}^{(3,m)}_4 = \widetilde a^{(3,m)}_{10} = \displaystyle \frac{8^{-3 - h}}{-4+2^h}(1152 + 15\cdot2^{6 + h} + 133\cdot4^{1 + h} + 89\cdot8^h - 39\cdot16^h)$ \\ \\
$\widetilde {a}^{(3,m)}_5 = \widetilde a^{(3,m)}_9 = \displaystyle \frac{4^{-5 - h}}{-4+2^h}(128 + 2^{6 + h} - 123\cdot4^{1 + h} - 123\cdot8^h)$ \\ \\
$\widetilde {a}^{(3,m)}_6 = \widetilde a^{(3,m)}_8 = -\displaystyle \frac{8^{-3 - h}}{-4+2^h}(640 + 9\cdot2^{6 + h} - 81\cdot2^{1 + 4\cdot h} + 105\cdot4^{1 + h} + 577\cdot8^h)$  \\ \\
$\widetilde {a}_7 = -\displaystyle \frac{4^{-4 - h}}{-4+2^h}(128 + 3\cdot2^{6 + h} + 81\cdot4^{1 + h} - 175\cdot8^h)$ \\ \\
\end{tabular}}

\noindent
where $h = 3+m^{-\mu}$.
The biorthogonal mask coefficients $\{\widetilde a^{(3,m)}_\alpha\}$ (rounded to the forth digit) are listed in Tab.~\ref{Tab_2}.
In Fig.~\ref{Fig_3} the behavior of $\{\widetilde a^{(3,m)}_\alpha\}$ and $\widetilde \varphi^{(3,0)}$ are displayed.
The biorthogonal wavelet mask coefficients ${\bf q}^{(3,m)}$ and $\widetilde {\bf q}^{(3,m)}$ can be obtained by 
$
q^{(3,m)}_\alpha = (-1)^{\alpha} \, \widetilde a^{(3,m)}_{-\alpha+1}$,
$\widetilde q^{(3,m)}_\alpha = - (-1)^{\alpha} \, a^{(3,m)}_{-\alpha+1}
$.
In Fig.~\ref{Fig_4} the behavior of $\psi^{(3,0)}$ and $\widetilde \psi^{(3,0)}$ is displayed.

\begin{table}
\caption{Numerical values (rounded to the forth digit)  of the mask coefficients
$\widetilde {a}^{(3,m)}_\alpha$, for $\alpha=0,\ldots,7$, and $m=0,\ldots,8$}
\label{Tab_2}
{\footnotesize 
\begin{tabular}{|c|ccccccccc|}
\hline
$m$ & 0 & 1 & 2 & 3 & 4 & 5 & 6 & 7 & 8 \\
\hline
 $\widetilde {a}^{(3,m)}_0 $ & 0.5 &  0.0011 &  0.0021 &  0.0026 &  0.0030 &  0.0064 &  0.0034 &  0.0035 &  0.0036 \\
 $\widetilde {a}^{(3,m)}_1 $ & 0.5 & -0.0085 & -0.0114 & -0.0129 & -0.0138 & -0.0288 & -0.0148 & -0.0151 & -0.0154 \\
 $\widetilde {a}^{(3,m)}_2 $ &  0  &  0.0066 &  0.0028 &  0.0005 & -0.0010 & -0.0039 & -0.0027 & -0.0032 & -0.0036 \\
 $\widetilde {a}^{(3,m)}_3 $ &  0  &  0.0574 &  0.0760 &  0.0857 &  0.0914 &  0.1905 &  0.0979 &  0.0999 &  0.1014 \\
 $\widetilde {a}^{(3,m)}_4 $ &  0  & -0.0810 & -0.0790 & -0.0768 & -0.0752 & -0.1480 & -0.0732 & -0.0725 & -0.0720 \\
 $\widetilde {a}^{(3,m)}_5 $ &  0  & -0.1998 & -0.5108 & -0.2834 & -0.2998 & -0.6211 & -0.3180 & -0.3236 & -0.3278 \\
 $\widetilde {a}^{(3,m)}_6 $ &  0  &  0.3233 &  0.3241 &  0.3237 &  0.3232 &  0.6456 &  0.3225 &  0.3222 &  0.3220 \\
 $\widetilde {a}^{(3,m)}_7 $ &  0  &  0.8019 &  0.8816 &  0.9212 &  0.9443 &  1.9187 &  0.9698 &  0.9776 &  0.9835 \\
 \hline
\end{tabular}
}
\end{table}

\begin{figure}
\centering
\begin{tabular}{ccc}
\includegraphics[width=5cm]{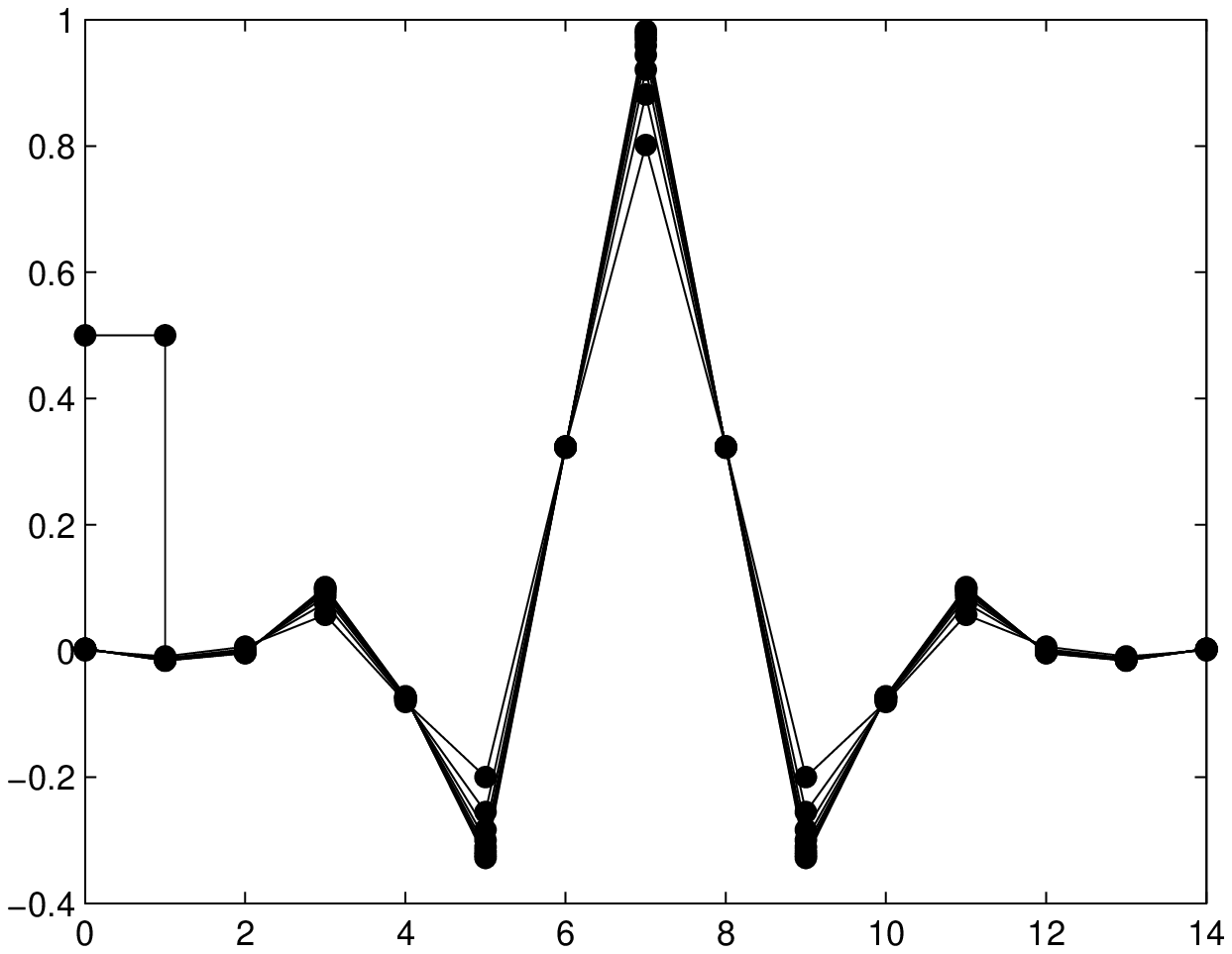}
& \quad &
\includegraphics[width=5cm]{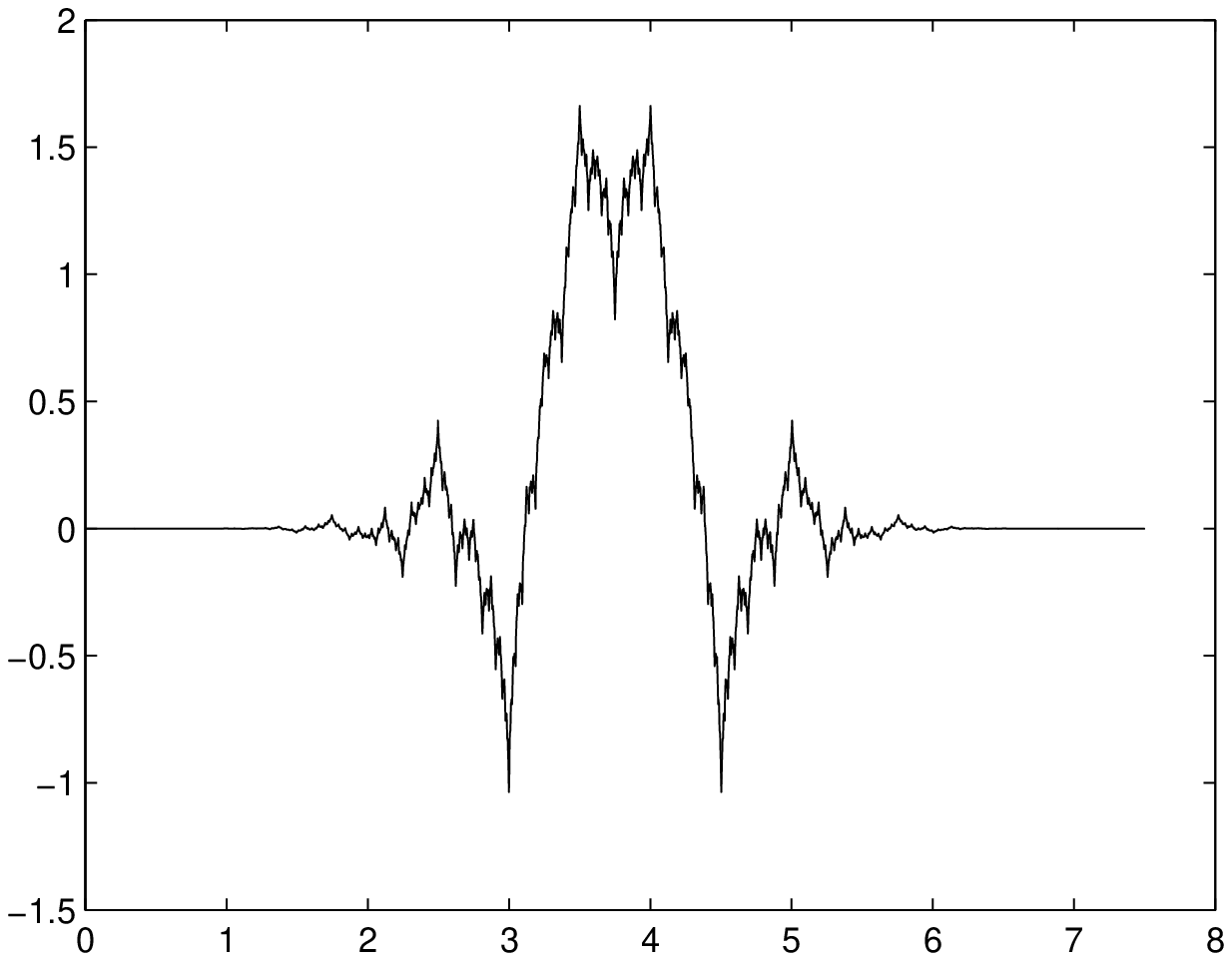}
\end{tabular}
\caption{The first 8 nonstationary biorthogonal masks 
$\widetilde {\bf a}^{(3,m)}$ (left) and $\widetilde \varphi^{(3,0)}$ (right)}
\label{Fig_3}
\end{figure}
Just to show how the properties of the constructed nonstationary biorthogonal filters can affect the analysis of a given signal, 
we evaluate the coefficients $\{ \lambda_\alpha^m \}$ and $\{ \zeta_\alpha^m \}$, obtained after three steps of the
decomposition algorithm (\ref{dec_alg}), when the starting sequence is a spike-like signal (see Fig.~\ref{Fig_5} (left)).
The coefficients are plotted in Fig.~\ref{Fig_5} (right) in comparison with the coefficients obtained when using the stationary
cubic spline biorthogonal filters. The figure shows that the nonstationary decomposition algorithm has higher compression properties: 
actually the number of nonzero coefficients are 26 in the nonstationary case, while they are 39 in the B-spline case.  

\begin{figure}
\centering
\begin{tabular}{ccc}
\includegraphics[width=5cm]{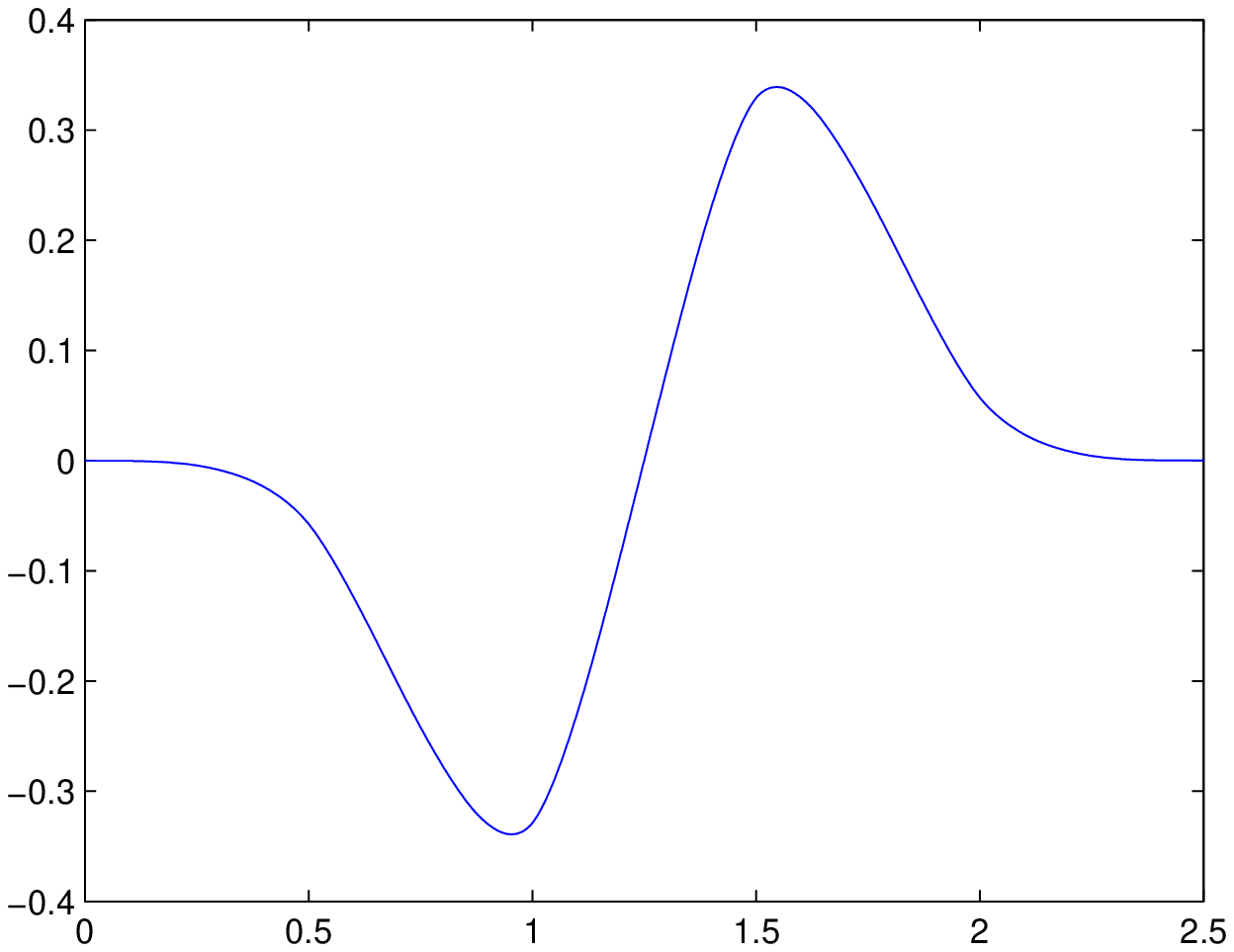}
& \quad &
\includegraphics[width=5cm]{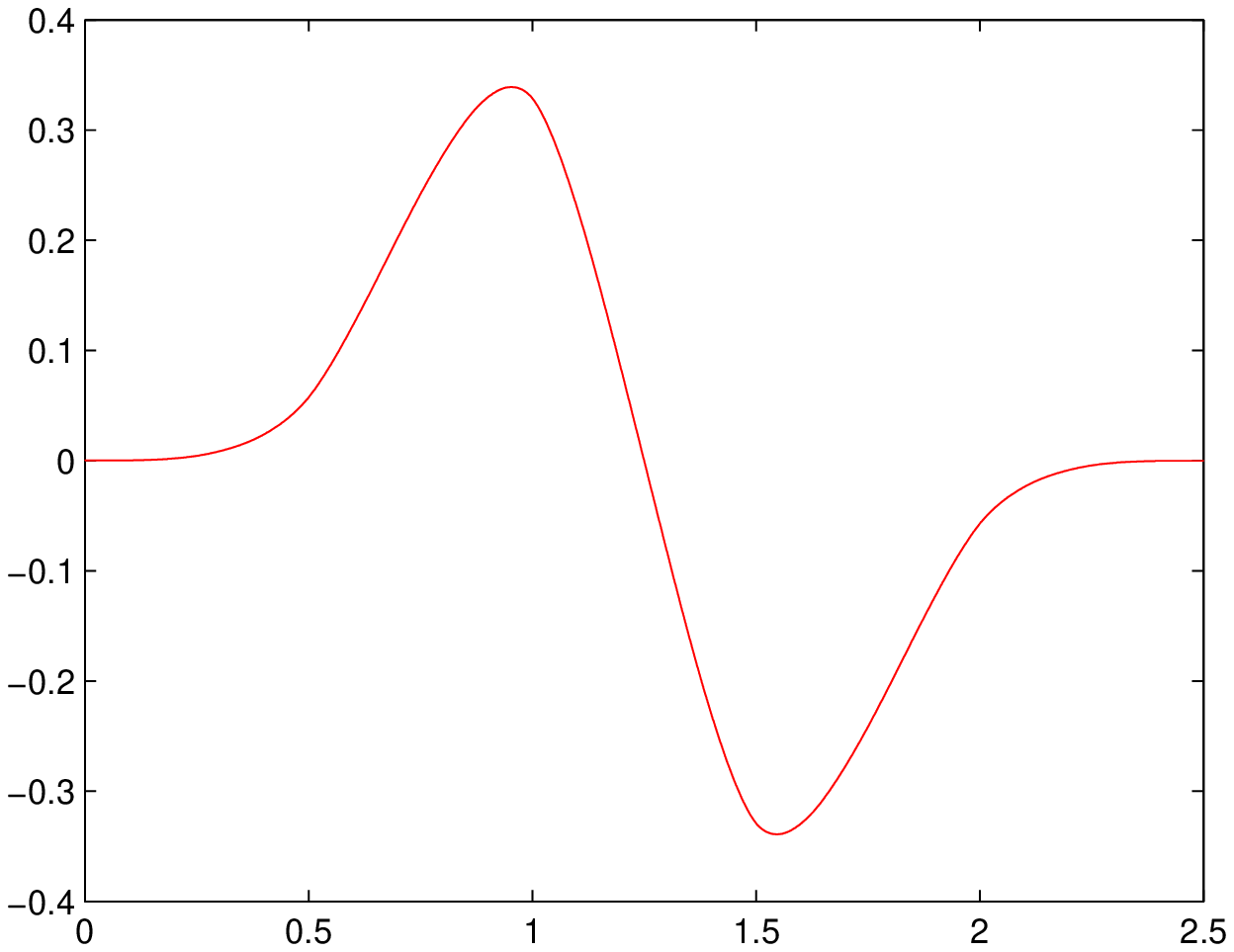}
\end{tabular}
\caption{Graphs of $\psi^{(3,0)}$ (left) and $\widetilde \psi^{(3,0)}$ (right)}
\label{Fig_4}
\end{figure}

\begin{figure}
\centering
\begin{tabular}{ccc}
\includegraphics[width=5cm]{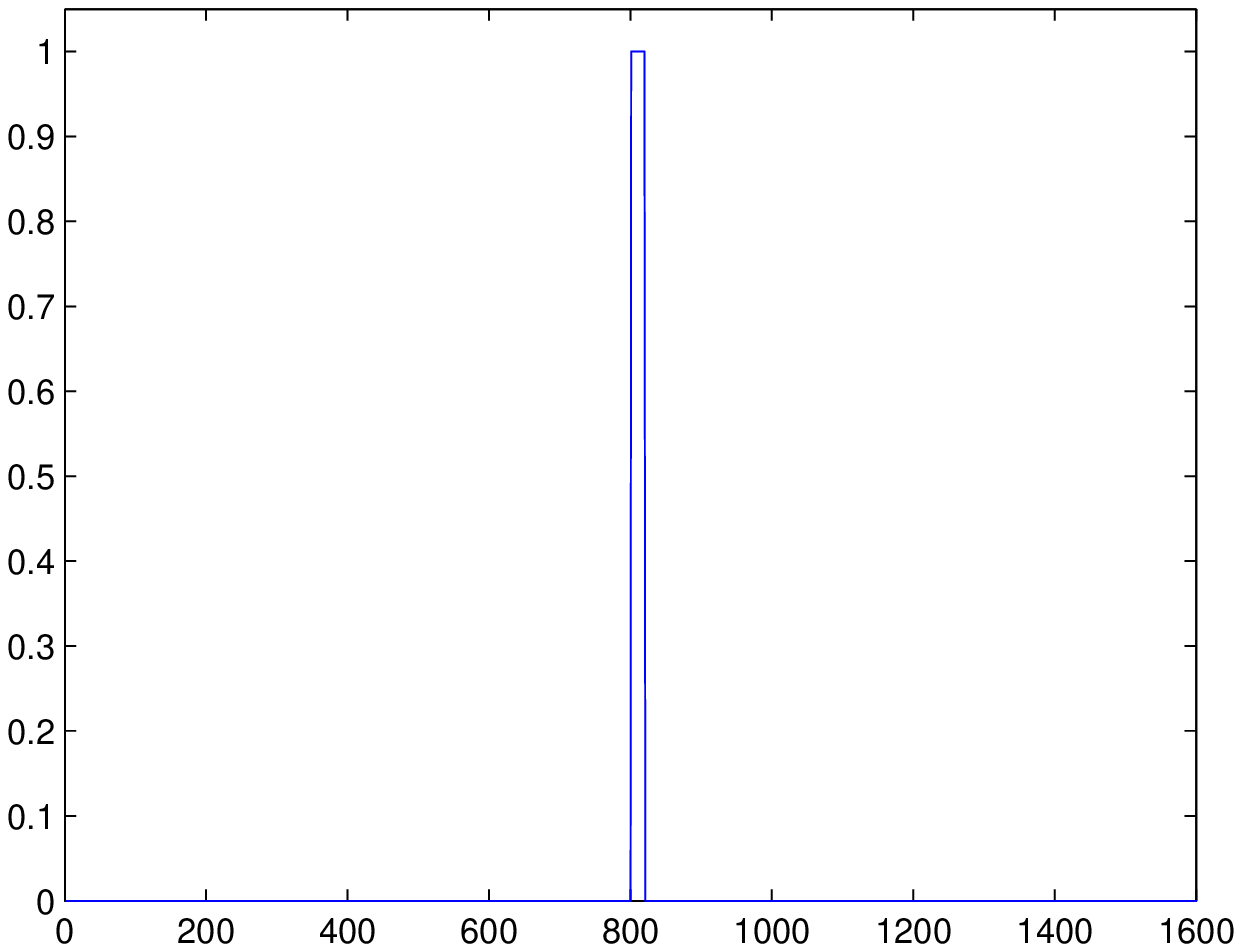}
& \quad &
\includegraphics[width=5cm]{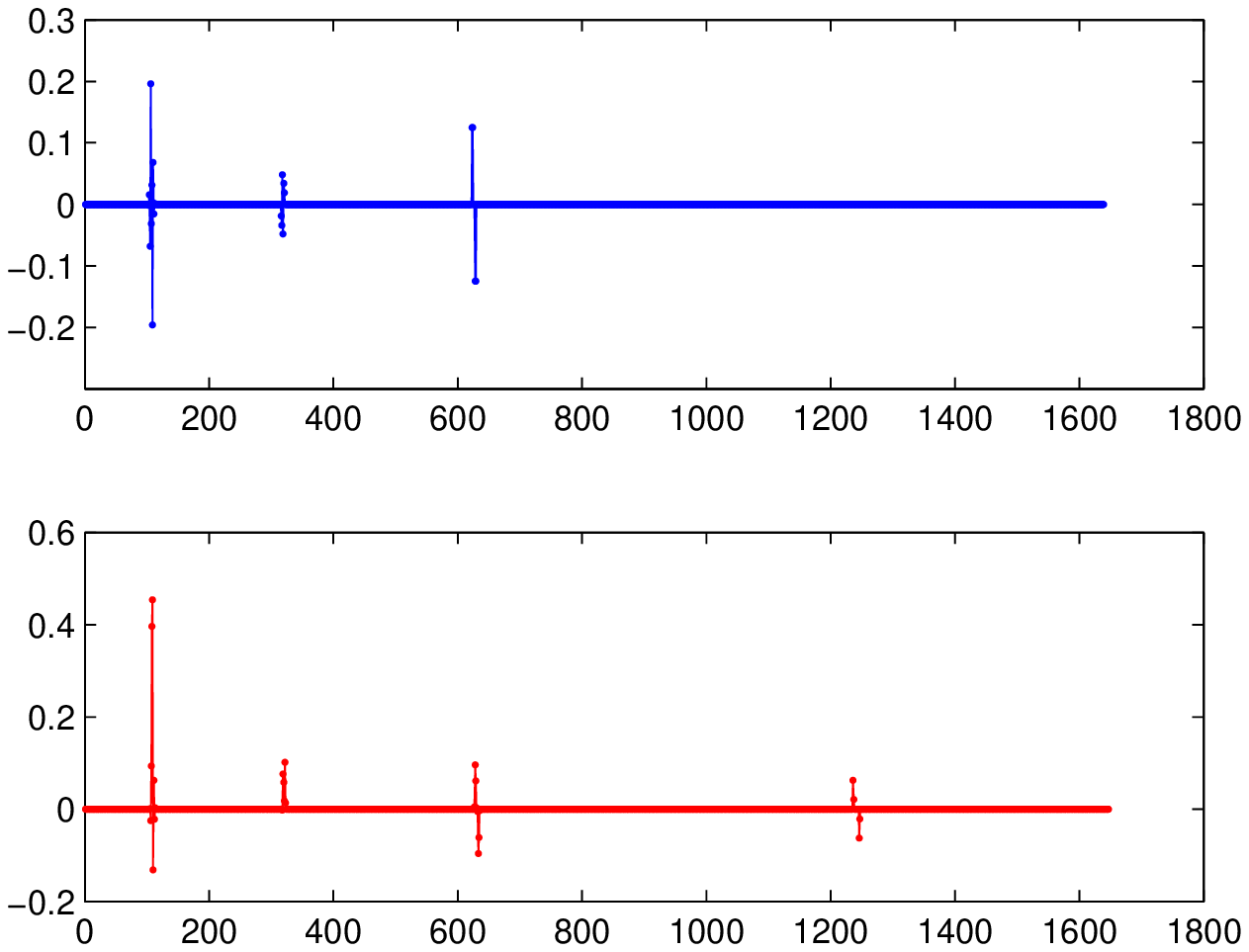}
\end{tabular}
\caption{The spike-like initial sequence  (left) and the decomposition coefficients obtained after 3 steps of the nonstationary decomposition algorithm (upper right). The decomposition coefficients obtained by the stationary cubic biorthogonal filters are also shown (bottom right)}
\label{Fig_5}
\end{figure}

\section{Conclusion}

We studied the properties of a class of refinable ripplets associated with sequences of nonstationary scaling masks. 
One of the most interesting property of these functions is in that they have a smaller support than 
the stationary refinable ripplets with the same smoothness. This localization property is crucial in several applications,
from geometric modeling to signal processing. 
\\
After proving some approximation properties, such as Strang--Fix conditions, polynomial reproduction and approximation order, 
we proved also that any refinable function in the family is bell-shaped, so that they can efficiently approximate a Gaussian.
Moreover, since these refinable functions generate nonstationary multiresolution analyses, we constructed the minimally supported nonstationary prewavelets and proved that their $2^m$-shifts form $L_2$-stable bases. 
We note that this construction can be generalized to other class of nonstationary refinable functions, like exponential splines.
Moreover, we constructed nonstationary biorthogonal bases and filters to be used in efficient decomposition and reconstruction algorithms.
\\
The localization property of the refinable ripplets we studied implies that the corresponding nonstationary wavelets have a small support too, a property which is very desirable in the case when the relevant information of a function to be approximated or of a signal to be analyzed are focused in small regions of the scale-time plane.
The preliminary test in Section 6 shows the good performances of the constructed nonstationary wavelets in a simple compression test. More tests will be the subject of a forthcoming paper.


%\bibliography{BiblNS}
%\bibliographystyle{plain}

\end{document}